\documentclass[a4paper,11pt,reqno]{amsart}
\usepackage{enumerate}
\usepackage[pdftitle={Rank connectivity and pivot-minors of graphs},pdfauthor={Sang-il Oum}]{hyperref}
\usepackage[a4paper]{geometry}
\usepackage{cite}
\usepackage{tikz}
\newtheorem{theorem}{Theorem}[section]
\newtheorem*{MAIN}{Theorem~\ref{thm:main}}
\newtheorem{lemma}[theorem]{Lemma}
\newtheorem{proposition}[theorem]{Proposition}

\newtheorem{corollary}[theorem]{Corollary}

\newcommand\rwd{\operatorname{rwd}}

\newcommand\A{\operatorname{\mbox{\boldmath $A$}}}

\newcommand\pivot{\wedge}
\newcommand\rank{\operatorname{rank}}
\newcommand\slantfrac[2]{\hbox{$\,^{#1}\!/_{#2}$}}
\newcommand\abs[1]{\lvert#1\rvert}
\newcommand\rc[2]{$#1^{+#2}$-rank-connected}
\begin{document}
\title{Rank connectivity and pivot-minors of graphs}
\author{Sang-il Oum}
\address{\small Discrete Mathematics Group, Institute~for~Basic~Science~(IBS), Daejeon,~South~Korea.}
\address{\small Department of Mathematical Sciences, KAIST,  Daejeon,~South~Korea.}
\thanks{Supported by the Institute for Basic Science (IBS-R029-C1).}
\email{sangil@ibs.re.kr}
\date{October 11, 2022}

\begin{abstract}
  The cut-rank of a set $X$ in a graph $G$ is the rank of the $X\times (V(G)-X)$ submatrix of the adjacency matrix over the binary field.
  A \emph{split} is a partition of the vertex set into two sets $(X,Y)$ such that the cut-rank of $X$ is less than $2$ and both $X$ and $Y$ have at least two vertices.
  A graph is \emph{prime} (with respect to the split decomposition) if it is connected and has no splits.
  A graph $G$ is $k^{+\ell}$-rank-connected if for every set $X$ of vertices with the cut-rank less than $k$, $\lvert X\rvert$ or $\lvert V(G)-X\rvert $ is less than $k+\ell$.
  We prove that
  every prime $3^{+2}$-rank-connected graph $G$ with at least $10$ vertices
  has a prime $3^{+3}$-rank-connected
  pivot-minor $H$ such that $\lvert V(H)\rvert =\lvert V(G)\rvert -1$.
  As a corollary,
  we show that every excluded pivot-minor for the class of graphs of rank-width
  at most $k$ has at most $(3.5 \cdot 6^{k}-1)/5$ vertices
  for $k\ge 2$.
  We also show that the excluded pivot-minors for the class of graphs
  of rank-width at most $2$ have at most $16$ vertices.
\end{abstract}
\maketitle

\section{Introduction}
\label{sec:intro}
For a subset $X$ of vertices of a graph $G$, the \emph{cut-rank}
function $\rho_G(X) $ is the rank of an $\abs{X}\times (\abs{V(G)}-\abs{X})$
matrix over $\mathrm{GF}(2)$ whose rows and columns are indexed by
vertices in $X$ and $V(G)-X$, respectively, and the entry is
$1$ if and only if the vertices corresponding to the row and the column, respectively,
are adjacent in $G$.

For integers $k$ and $l$,
a graph $G$ is called \emph{\rc{k}{\ell}} if
for every subset $X$ of $V(G)$,
$\min(\abs{X},\abs{V(G)-X})<k+\ell$ whenever $\rho_G(X)<k$.
A graph is \emph{$k$-rank-connected} if
it is \rc{m}{0}{} for all integers $m\le k$.\footnote{
  Unlike the usual definition, 
  the definition of \rc{k}{\ell}ness does not require $k'$-rank-connectedness for $k'<k$. Our intention is to make some statements, in particular Proposition~\ref{prop:rankconn}, easier to be stated.
}
The \emph{rank connectivity} of a graph $G$
is the maximum  integer $k$ such that
$G$ is $k$-rank-connected.
All graphs are $0$-rank-connected
and the $1$-rank-connected graphs are exactly the connected graphs.
The $2$-rank-connected graphs are exactly \emph{prime} graphs with
respect to split decompositions. Prime graphs have been studied by
several researchers, notably~\cite{Cunningham1982,Bouchet1987b,Spinrad1989}.

We are interested in proving a ``chain theorem,'' which
guarantees the existence of a highly-connected large substructure in a `well-connected'
structure.
Such theorems are useful tools to be used with induction.

We prove a theorem
for prime \rc{3}{2}{} graphs to find a prime \rc{3}{3}{} pivot-minor with one less number of vertices.
We will define pivot-minors and pivot-equivalence in Section~\ref{sec:def}.

\begin{MAIN}
  If   $G$ is a prime \rc{3}{2}{} graph
  with at least $10$ vertices, then
  $G$ has a prime \rc{3}{3}{} pivot-minor $H$
  such that $\abs{V(H)}=\abs{V(G)}-1$.
\end{MAIN}

For one of its applications,
we present a result on the size of excluded pivot-minors for graphs of rank-width $k$.
Rank-width is a width parameter of graphs introduced by Oum and
Seymour~\cite{OS2004}.
An \emph{excluded pivot-minor} for graphs of rank-width at most $k$
is a pivot-minor-minimal graph having rank-width larger than $k$.
As a corollary of Theorem~\ref{thm:main},
we prove the following in Section~\ref{sec:rwd}.
\begin{itemize}
  \item Every excluded pivot-minor
        for graphs of rank-width at most $k$
        has at most  $(\slantfrac{7}{12}\cdot 6^{k+1}-1)/5$ vertices for $k\ge 2$.
  \item Every excluded pivot-minor for graphs of rank-width at most $2$
        has at most $16$ vertices.
\end{itemize}
These are improvements of a previous result~\cite{Oum2004},
showing that
such graphs
have at most $(6^{k+1}-1)/5$ vertices.
For graphs of rank-width $1$, the list of excluded pivot-minors and
vertex-minors are known exactly, see Oum~\cite{Oum2016}, but the list is not known for graphs
of rank-width $2$.
Now it may be possible to search all graphs up to $16$ vertices to
determine the list of excluded pivot-minors for the class of graphs of rank-width at most $2$.

There are other chain theorems known in other contexts.
One of the  well-known chain theorems is
Tutte's Wheels and Whirls Theorem for matroids.
Allys~\cite{Allys1994}
proved the following chain theorem on $2$-rank-connected graphs,
which can be considered as a generalization of Tutte's Wheels and Whirls Theorem in a certain sense.
(See Geelen
\cite{Geelen1995} for an alternative proof.)
\begin{theorem}[{Allys~\cite[Theorem 4.3]{Allys1994}}]\label{thm:allys}
  Let $G$ be a prime graph with at least $5$ vertices.
  Then $G$ has a prime pivot-minor $H$ such that $\abs{V(H)}=\abs{V(G)}-1$,
  unless
  $G$ is pivot-equivalent to a cycle.
\end{theorem}

Our goal is to explore  chain theorems to higher rank connectivity.
Our main theorem is motived by the following theorem on
internally $4$-connected matroids by Hall~\cite{Hall2005}.
For the definitions on matroids, we refer to the book of Oxley~\cite{Oxley1992}.
\begin{theorem}[{Hall~\cite[Theorem 3.1]{Hall2005}}]\label{thm:hall}
  Let $M$ be an internally $4$-connected matroid, and let
  $\{a,b,c\}$ be a triangle of $M$. Then at least one of the following hold.
  \begin{enumerate}[(1)]
    \item At least one of $M\setminus a$, $M\setminus b$ and $M\setminus c$ is
          $4$-connected up to $3$-separators of size $4$.
    \item At least two of $M\setminus a$, $M\setminus b$ and $M\setminus c$ are
          $4$-connected up to $3$-separators of size $5$.
  \end{enumerate}
\end{theorem}
The above theorem has been used later in several papers~\cite{GZ2006,GZ2008} to prove stronger theorems for internally $4$-connected matroids.

Why is it possible to find theorems on the rank connectivity of graphs
analogous to those on the matroid connectivity?
We do not need matroids for the proofs in the paper, but there are interesting connections between
bipartite graphs and binary matroids, fully discussed in~\cite{Oum2004}.
Let $M$ be the binary matroid  on a finite
set $E=X\cup Y$ with a binary representation
\[
  \bordermatrix{   & X & Y\cr
    X &
    \begin{matrix}
      1                      \\
       & 1                   \\
       &   & \ddots          \\
       &   &        & 1\quad
    \end{matrix}
    \vline
    & \qquad A \qquad
  }.\]
Then in the matroid $M$, the \emph{rank} $r_M(S)$ of a set $S$ is the dimension of the vector
space spanned by column vectors indexed by $S$.
The \emph{connectivity function} $\lambda_M(S)$  of $M$ is
defined as
$\lambda_M(S)=r_M(S)+r_M(E-S)-r_M(E)$
for all $S\subseteq E$.
A matroid $M$ is \emph{$n$-connected} if for all integers $k<n$,
$\min(\abs{X}, \abs{E(M)-X})<k$
whenever $\lambda_M(X)<k$.
The \emph{fundamental graph} $G$ of $M$
is a bipartite graph with the bipartition $(X,Y)$
such that  $i$ in $X$ is adjacent to a vertex $j$ in $B$
if and only if the $(i,j)$-entry of $A$ is non-zero.
Because  $\rho_G(S)=\lambda_M(S)$ for all $S\subseteq E$ (see Oum~\cite{Oum2004}),
the binary matroid $M$ is $k$-connected if and only if
its fundamental graph $G$ is $k$-rank-connected.
Minors of binary matroids and pivot-minors of $G$ are also related~\cite{Oum2004}.
We also remark that internally $4$-connected binary matroids correspond to
prime \rc{3}{1} bipartite graphs.
Thus it is natural to consider generalizations of theorems on binary matroids regarding the matroid connectivity to non-bipartite graphs with the rank connectivity.

Here is an overview of the paper. In Section~\ref{sec:def}, we discuss
submodularity and graph pivot-minors.
Section~\ref{sec:gen2} proves a useful proposition to generate
prime pivot-minors.
Section~\ref{sec:main} states our main theorem
and its proof for two easy cases.
Sections~\ref{sec:triplet} and~\ref{sec:internally4prime} provide the
proof of the remaining cases of the main theorem.
Section~\ref{sec:rwd} discusses an application to the problem of bounding the size of excluded
pivot-minors for the class of graphs of rank-width at most $2$.

\section{Preliminaries: Submodularity and Pivoting}
\label{sec:def}
For an $X\times Y$ matrix $M$
and
subsets $U$ of $X$ and $V$ of $Y$,
we write $M[U,V]$ to denote the $U\times V$
submatrix of $M$. The following proposition is well known; for instance, see~\cite[Proposition  2.1.9]{Murota2000},
\cite[Lemma 2.3.11]{Truemper1992}, or~\cite{Truemper1985}.
\begin{proposition}\label{prop:submodular}
  Let $M$ be an $X\times Y$ matrix.
  Let $X_1,X_2\subseteq X$ and $Y_1,Y_2\subseteq Y$.
  Then
  \[\rank(M[X_1,Y_1])+\rank(M[X_2,Y_2])
    \ge \rank(M[X_1\cap X_2,Y_1\cup Y_2])
    +\rank(M[X_1\cup X_2,Y_1\cap Y_2]).\]
\end{proposition}

All graphs in this paper are simple.
For a graph $G=(V,E)$,
Let $\A_G$ be the adjacency matrix of a graph $G$  over the binary
field $\mathrm{GF}(2)$, that is a $V\times V$ matrix such that the
$(i,j)$ entry is $1$ if and only if $i$ is adjacent to $j$ in $G$.
For subsets $X$ and $Y$ of $V$,
Let $\rho_G(X,Y)=\rank(\A_G[X,Y])$.
The \emph{cut-rank function} $\rho_G(X)$ is $\rho_G(X,V(G)-
  X)$.
We will sometimes write $\rho$ for $\rho_G$ if it is not ambiguous.
A partition $(A,B)$ of the vertex set of a graph $G$
is called a \emph{split}
if $\rho_G(A)\le 1$
and $\min(\abs{A},\abs{B})\ge 2$.

By Proposition~\ref{prop:submodular}, one can easily show the following
two lemmas, which will be very useful later.

\begin{lemma}\label{lem:gagb}
  Let $G$ be a graph and let $a$, $b$ be distinct vertices.
  Let $A\subseteq V(G)- \{a\}$ and $B\subseteq V(G)-
    \{b\}$.
  If $b\notin A$ and $a\notin B$,
  then
  \begin{equation}
    \label{eq:gagb1}
    \rho_G(A\cap B)+\rho_{G\setminus a\setminus b}(A\cup B)\le
    \rho_{G\setminus a}(A)+\rho_{G\setminus b}(B).
  \end{equation}
  If $b\in A$ and $a\notin B$,
  then
  \begin{equation}
    \label{eq:gagb2}
    \rho_{G\setminus b}(A\cap B)+\rho_{G\setminus a}(A\cup B)\le
    \rho_{G\setminus a}(A)+\rho_{G\setminus b}(B).
  \end{equation}
  If $b\in A$ and $a\in B$,
  then
  \begin{equation}
    \label{eq:gagb3}
    \rho_{G\setminus a\setminus b}(A\cap B)+\rho_{G}(A\cup B)\le
    \rho_{G\setminus a}(A)+\rho_{G\setminus b}(B).
  \end{equation}
\end{lemma}
\begin{lemma}\label{lem:gga}
  Let $G$ be a graph and let $x$ be a vertex of $G$.
  Let $X,Y\subseteq V(G)-\{x\}$.
  Then
  \begin{align}\label{eq:gga1}
    \rho_{G\setminus x}(X\cap Y)+\rho_G(X\cup Y\cup \{x\})
     & \le \rho_{G\setminus x}(X)+\rho_G(Y\cup \{x\}), \\
     \label{eq:gga2}
    \rho_{G}(X\cap Y)+\rho_{G\setminus x}(X\cup Y)
     & \le \rho_{G\setminus x}(X)+\rho_G(Y).
  \end{align}
\end{lemma}

We defined \rc{k}{\ell}{} graphs in Section~\ref{sec:intro}.
\begin{proposition}\label{prop:rankconn}
  Let $k$ be a positive integer.
  Let $G=(V,E)$ be a \rc{k}{0}{} graph with at least $2k$ vertices.
  Let $v\in V$.
  Then both $G$ and $G\setminus v$ are \rc{(k-1)}{0}{}.

  Consequently if $G$ is $k$-rank-connected and $\abs{V}\ge 2k$,
  then $G\setminus v$ is $(k-1)$-rank-connected for all $v\in V$.
\end{proposition}
\begin{proof}
  Let us first prove that $G\setminus v$ is \rc{(k-1)}{0}{}
  for each $v\in V$.
  Let $(X,Y)$ be a partition of $V-\{v\}$ such that
  $\abs{X},\abs{Y}\ge k-1$ and
  $\rho_{G\setminus v}(X)<k-1$.
  We deduce that $\abs{X}<k$ because $\rho(X,Y\cup\{v\})<k$.
  Similarly $\abs{Y}<k$ because $\rho(X\cup\{v\},Y)<k$.
  Then $\abs{V}=\abs{X}+\abs{Y}+1<2k$, a contradiction.

  Now let us assume that  $G$ is not \rc{(k-1)}{0}.
  Suppose that $(X,Y)$ be a partition of $V$ such that
  $\abs{X}, \abs{Y}\ge k-1$ and $\rho_G(X)<k-1$.
  We may assume that $\abs{X}\le \abs{Y}$.
  Since $\abs{V}\ge 2k$, $\abs{Y}\ge k$.
  We pick $v\in Y$. Then $\abs{Y-\{v\}}\ge k-1$ and $\rho_{G\setminus
      v}(X)<k-1$; a contradiction because $G\setminus v$ is
  \rc{(k-1)}{0}.
\end{proof}
For a vertex $v$ of a graph $G$,
let $N(v)$ be the set of neighbors of $v$
and let $N[v]=N(v)\cup \{v\}$.

\begin{figure}\label{fig:pivot}
  \tikzstyle{v}=[circle, draw, solid, fill=black, inner sep=0pt, minimum width=2pt]
  \tikzstyle{w}=[circle, draw, solid, fill=black, inner sep=0pt, minimum width=4pt]

  \centering
  \begin{tikzpicture}
    \begin{scope}[xshift=-3cm]
      \node [v,label=$x$] at (-2,2) (x) {} ;
      \node [v,label=$y$] at (2,2) (y) {} ;
      \draw (x)-- (0,1.8) node [w] (c1) {} --(y);
      \draw (x)-- (0,1.5) node [w] (c2) {} --(y);
      \draw (x)-- (-1,1) node [w] (l1) {} ;
      \draw (x)--(-1,0.5) node [w] (l2) {};
      \draw (y)--(1,1) node [w] (r1){};
      \draw (y)--(1,0.5) node [w] (r2){};
      \draw (l1)--(0,0) node[v] (z){} -- (r1);
      \draw (l2)--(z)--(c2);
      \draw (x) [out=10,in=170] to (y);
      \draw (l1)--(l2); 
      \draw [thick] (l1)--(c1)--(l2)--(c2)--(r1)--(l1);
      \draw [thick] (l2)--(r2);
      \draw (0,-.5) node {$G$};
    \end{scope}
    \begin{scope}[xshift=3cm]
      \node [v,label=$y$] at (-2,2) (x) {} ;
      \node [v,label=$x$] at (2,2) (y) {} ;
      \draw (x)-- (0,1.8) node [w] (c1) {} --(y);
      \draw (x)-- (0,1.5) node [w] (c2) {} --(y);
      \draw (x)-- (-1,1) node [w] (l1) {} ;
      \draw (x)--(-1,0.5) node [w] (l2) {};
      \draw (y)--(1,1) node [w] (r1){};
      \draw (y)--(1,0.5) node [w] (r2){};
      \draw (l1)--(0,0) node[v] (z){} -- (r1);
      \draw (l2)--(z)--(c2);
      \draw (x) [out=10,in=170] to (y);
      \draw (l1)--(l2); 
      \draw [thick] (l1)--(c2);
      \draw [thick] (l1)--(r2)--(c1)--(r1);
      \draw [thick] (l2)--(r1);
      \draw [thick] (c2)--(r1);
      \draw [thick] (c2)--(l1);
      \draw (0,-.5)node {$G\pivot xy$};
    \end{scope}
  \end{tikzpicture}
  \caption[]{An illustration of pivoting an edge $xy$ in a graph $G$.}
\end{figure}
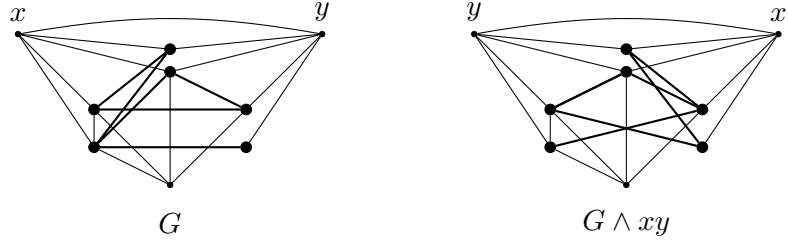

\emph{Pivoting} an edge $vw$ of a graph $G$ is an operation to 
obtain a new graph denoted by $G\pivot vw$ from $G$ by the following procedures:
\begin{enumerate}[(i)]
  \item For each $(x,y)\in (N(v)-N(w)-\{w\})\times (N(w)-N(v)-\{v\})$, 
  $(x,y)\in (N(v)-N(w)-\{w\})\times (N(w)\cap N(v))$, or $(x,y)\in (N(v)\cap N(w))\times (N(w)-N(v)-\{v\})$, we remove the edge $xy$ if $x$, $y$ are adjacent in $G$ and we add a new edge $xy$ if $x$, $y$ are non-adjacent.
  \item Swap the labels of $v$ and $w$; in other words, make $N_G(w)$ the set of neighbors of $v$ in $G\pivot vw$, and make $N_G(v)$ the set of neighbors of $w$ in $G\pivot vw$.
\end{enumerate}
See Figure~\ref{fig:pivot} for an illustration.

A graph $H$ is \emph{pivot-equivalent} to $G$
if $H$ is obtained from $G$ by a sequence of pivots.
A graph $H$ is a \emph{pivot-minor} of $G$
if $H$ is an induced subgraph of a graph
pivot-equivalent to $G$.
Pivots preserve the cut-rank functions.
\begin{lemma}[see~{\cite[Proposition 2.6]{Oum2004}}]\label{lem:pivotcutrank}
  If two graphs $G$ and $H$ are pivot-equivalent, then 
  $\rho_G(X)=\rho_H(X)$ for all sets $X$ of vertices.
\end{lemma}

It is easy to verify that  if $y_1,y_2\in N(x)$, then
$G\pivot xy_1=(G\pivot xy_2)\pivot y_1y_2$; see~\cite[Proposition
  2.5]{Oum2004}.
Since  $(G\pivot xy_1)\setminus x$ is pivot-equivalent to
$(G\pivot xy_2)\setminus x$,
we often write $G/x$ to denote $(G\pivot xy)\setminus x$
for some neighbor $y$ of
$x$, when there is no confusion. (If $x$ has no neighbors, we let
$G/x=G\setminus x$.)
We have to pay attention to cause no confusions; for instance,  if
$y,z\in N(x)$, then
$G/x\setminus y$ is somewhat ambiguous
because
$(G\pivot xy)\setminus x\setminus y$
and $(G\pivot xz)\setminus x\setminus y$
are not always pivot-equivalent.

The following lemma is in~\cite{Oum2004}.
\begin{lemma}[{\cite[Lemma 4.4]{Oum2004}}]\label{lem:bixbyineq}
  Let $G$ be a graph. Let $v$ be a vertex in $G$.
  Let $(X_1,X_2)$ and $(Y_1,Y_2)$ be partitions of $V(G)-
    \{v\}$. Then
  \[
    \rho_{G\setminus v}(X_1)+\rho_{G/v}(Y_1)
    \ge
    \rho_{G}(X_1\cap Y_1)
    +\rho_G(X_2\cap Y_2)-1.
  \]
\end{lemma}
\begin{proof}
  In~\cite[Proposition 4.3]{Oum2004}, it is shown that
  $\rho_{G/v}(Y_1)=\rho_G(Y_1\cup \{v\}, Y_2\cup\{v\})-1$
  if $v$ has at least one neighbor.
  So if $v$ is not an isolated vertex, then we obtain the desired
  inequality from Proposition~\ref{prop:submodular}.
  If $v$ is an isolated vertex, then $G/v=G\setminus v$
  and therefore
  $\rho_{G\setminus v}(X_1)+\rho_{G\setminus v}(Y_1)
    \ge \rho_{G\setminus v}(X_1\cap Y_1)+
    \rho_{G \setminus v}(X_2\cap Y_2)
    >\rho_{G}(X_1\cap Y_1)+
    \rho_G(X_2\cup Y_2)-1$ by Proposition~\ref{prop:submodular}.
\end{proof}
By Lemma~\ref{lem:bixbyineq}, we deduce the following,
motivated by Bixby~\cite[Theorem 1]{Bixby1982}.
\begin{proposition}\label{prop:bixby}
  Let $G$ be a \rc{k}{\ell}{} graph
  and let $v$ be a vertex in $G$.
  Then $G\setminus v$ or $G/v$ is \rc{k}{(2\ell+k-1)}.
\end{proposition}
\begin{proof}
  Suppose that both $G\setminus v$ and $G/v$ are not
  \rc{k}{(2\ell+k-1)}.
  There exist partitions $(X_1,X_2)$ and $(Y_1,Y_2)$ of $V(G)-
    \{v\}$ such that
  $\rho_{G\setminus v}(X_1)<k$, $\rho_{G/v}(Y_1)<k$
  and $\abs{X_1},\abs{X_2},\abs{Y_1},\abs{Y_2}\ge 2k+2\ell-1$.
  By swapping $Y_1$ and $Y_2$ if necessary,
  we may assume that $\abs{X_1\cap Y_1}\ge k+\ell$.
  Since $G$ is \rc{k}{\ell}{}, $\rho_G(X_1\cap Y_1)\ge k$.
  By Lemma~\ref{lem:bixbyineq},
  $\rho_{G}(X_2\cap Y_2)<k$.
  Since $G$ is \rc{k}{\ell}, $\abs{X_2\cap Y_2}<k+\ell$.

  By applying Lemma~\ref{lem:bixbyineq} with $(X_1,X_2)$ and
  $(Y_2,Y_1)$,
  we deduce that $\rho_G(X_1\cap Y_2)<k$ or
  $\rho_G(X_2\cap Y_1)<k$.
  If $\rho_G(X_1\cap Y_2)<k$, then $\abs{X_1\cap Y_2}<k+\ell$, and therefore
  $\abs{Y_2}=\abs{X_1\cap Y_2}+\abs{X_2\cap Y_2}<2k+2\ell-1$.
  Similarly, if $\rho_G(X_2\cap Y_1)<k$, then $\abs{X_2\cap Y_1}<k+\ell$, and therefore
  $\abs{X_2}<2k+2\ell-1$. This is a contradiction.
\end{proof}

\section{Generating prime pivot-minors}
\label{sec:gen2}
A set $X$ of vertices of a graph $G$ is \emph{fully closed}
if $\rho_G(X)=2$, $\abs{X}>2$, and 
$\rho_G(X\cup\{v\})>\rho_G(X)$
for all $v\in V(G)-X$.
Our goal in this section is to prove the following proposition.
\begin{proposition}\label{prop:fully}
  Let $G=(V,E)$ be a prime graph with $\abs{V}\ge 8$.
  If $G$ has a fully closed set $A$, 
  then $A$ has  a vertex $v$ such that 
  $G\setminus v$ or $G/v$ is
  prime.
\end{proposition}
We can not omit the condition that $A$ is fully closed; for instance, if $G$ is a cycle of length at least $6$,
then neither $G\setminus v$ nor $G/v$  is prime for a vertex~$v$ of $G$.

Let us discuss how to prove Proposition~\ref{prop:fully}.
By Proposition~\ref{prop:bixby},
for each vertex $v$ of $G$, $G\setminus v$ or $G/v$ is \rc{2}{1}.
The rest of this section is devoted to find a vertex $v$ so that
$G\setminus v$ or $G/v$ is prime.

\begin{lemma}\label{lem:intprime}
  Let $G=(V,E)$ be a prime  graph with $\abs{V}\ge 7$
  and let $a\in V$.
  Suppose that $G\setminus a$ is \rc{2}{1}{}
  and let $b$, $c$ be distinct vertices of $G\setminus a$
  such that $\rho_{G\setminus a}(\{b,c\})\le 1$.
  Suppose that
  $G\setminus b$ is \rc{2}{1}{} but not prime.
  \begin{enumerate}[(i)]
    \item
          If $c$ has a neighbor other than $a$ and $b$, then
          $G\setminus a\setminus b$ is prime.
    \item
          If $\abs{V}>7$ and $c$ is adjacent to only $a$ and $b$, then
          $G\setminus c$ is prime.
    \item
          If $\abs{V}=7$, $c$ is adjacent to only $a$ and $b$, and $G\setminus
            c$ is not prime, then
          $\abs{A}=\abs{B}=3$ for every split $(A,B)$ of $G\setminus c$,
          and
          $(G\pivot ac)\setminus a\setminus c$ is prime.
  \end{enumerate}
\end{lemma}
\begin{proof}
  Let us first prove (i).
  Since   $G\setminus b$ is not prime,
  there
  exists a set $X$ such that  $\rho_{G\setminus b}(X)\le 1$ and $\abs{X}\ge 2$,
  $\abs{V}-1-\abs{X}\ge 2$.
  We may assume that $c\in X$.
  If $a\in X$, then
  By (\ref{eq:gagb3}) of Lemma~\ref{lem:gagb} with
  $\rho_{G\setminus a}(\{b,c\})=1$,
  $\rho_{G\setminus b}(X)\le 1$,
  we deduce that
  $\rho_{G}(X\cup\{b\})\le 1$; a contradiction because  $G$ is
  prime.
  So $a\notin X$. By (\ref{eq:gagb2}) of Lemma~\ref{lem:gagb},
  we deduce that $\rho_{G\setminus a}(X\cup\{b\})\le 1$.
  Since $G\setminus a$ is \rc{2}{1}{}, $\abs{V- X}=2$.
  Let $V- X=\{a,d\}$.
  So $\rho_{G\setminus b}(\{a,d\})\le 1$.

  To prove that $G\setminus a\setminus b$ is prime,
  let us assume on the contrary
  that $G\setminus a\setminus b$ has a split $(A,B)$.
  We may assume that $c\in A$.

  From \eqref{eq:gga1} of Lemma~\ref{lem:gga},
  we deduce that
  $\rho_{G\setminus a\setminus b}(\{c\})+\rho_{G\setminus
      a}(A\cup\{b\})\le \rho_{G\setminus a\setminus
      b}(A)+\rho_{G\setminus a}(\{b,c\})\le 2$.
  Since $c$ has a neighbor other than $a$ and $b$,
  we have $\rho_{G\setminus a}(A\cup \{b\})\le 1$.
  Since $G\setminus a$ is \rc{2}{1}{},
  we have $\abs{B}=2$.

  If $d\in A$, then
  by \eqref{eq:gagb3} of Lemma~\ref{lem:gagb},
  \[ \rho_{G\setminus a\setminus b}(\{d\})+
    \rho_G(A\cup\{a,b \})\le \rho_{G\setminus a}(A\cup\{b\})
    +\rho_{G\setminus b}(\{a,d\})\le 2.\] 
  Since $\abs{A}\ge 2$ and $\abs{B}\ge 2$, we have $\rho_{G}(A\cup\{a,b\})\ge
    2$ and therefore $\rho_{G\setminus a\setminus b}(\{d\})=0$. Then
  by (\ref{eq:gga1}) of Lemma~\ref{lem:gga},
  $
    \rho_{G\setminus a}(\{b,c,d\})\le \rho_{G\setminus a\setminus
      b}(\{d\})
    +\rho_{G\setminus a}(\{b,c\})\le 1$.
  This contradicts to the assumption that $G\setminus a$ is
  \rc{2}{1}.

  Therefore, we may assume that $d\notin A$.
  We know that $\rho_{G\setminus a}(B)=\rho_{G\setminus a}(A\cup
    \{b\})\le 1$.
  By (\ref{eq:gagb2}) of Lemma~\ref{lem:gagb},
  $\rho_{G\setminus a}(\{d\})+\rho_{G\setminus b}(B\cup\{a\})\le
    \rho_{G\setminus a}(B)+\rho_{G\setminus b}(\{a,d\})\le 2$.
  Since $\abs{B\cup \{a\}}\ge 3$ and $G\setminus b$ is
  \rc{2}{1},
  we deduce that $\abs{A}=2$. Then, $\abs{V}\le \abs{A}+\abs{B}+2\le 6$, a
  contradiction.
  This proves that if $c$ has a neighbor other than $a$ and~$b$,
  then $G\setminus a\setminus b$  is prime.

  To prove (ii) and (iii), let us assume that $c$ has no neighbors
  other than $a$ and $b$.
  Suppose that $G\setminus c$ is not prime.
  Let $A$, $B$ be subsets of $V(G\setminus c)$ such that
  $B=V- (A\cup\{c\})$, $\abs{A}\ge 2$, $\abs{B}\ge 2$, and
  $\rho_{G\setminus c}(A)\le 1$.
  We may assume that $a\in A$. Since $G$ is prime, $\rho_G(A\cup \{a\})\ge 2$
  and therefore $b\in B$.

  From \eqref{eq:gagb2} of Lemma~\ref{lem:gagb},
  $\rho_{G\setminus c}(\{b\})+\rho_{G\setminus a}(B\cup\{c\})
    \le \rho_{G\setminus a}(\{b,c\})+\rho_{G\setminus c}(B)\le 2$.
  We deduce that   $\rho_{G\setminus a}(B\cup
    \{c\})\le 1$.
  Since  $G\setminus a$ is  \rc{2}{1},
  $\abs{A-\{a\}}\le 2$ and therefore $\abs{A}\le 3$.

  We also know that $\rho_{G\setminus b}(\{a,c\})=1$ because the
  degree of $c$ is $2$.
  From \eqref{eq:gagb2} of Lemma~\ref{lem:gagb},
  $\rho_{G\setminus c}(\{a\})+\rho_{G\setminus b}(A\cup\{c\})
    \le \rho_{G\setminus b}(\{a,c\})+\rho_{G\setminus c}(A)\le 2$.
  We conclude that $\abs{B-\{b\}}\le 2$ because $G\setminus b$ is
  \rc{2}{1}.

  Then $\abs{V}= \abs{A}+\abs{B}+1\le 7$.
  So we may assume that $\abs{V}=7$, $G\setminus c$ is not prime, and
  every split $(A,B)$ of $G$ satisfies $\abs{A}=\abs{B}=3$.
  We omit the remaining finite case checking to see
  why $(G\pivot ac) \setminus a\setminus c$ is prime.
\end{proof}
We will show that reducing by one vertex is enough
when there is a good fully closed set.
\begin{lemma}\label{lem:triplet1side}
  Let $G=(V,E)$ be a prime graph with $\abs{V}\ge 4$.
  Let $A$ be  a  fully closed set and let $v\in A$.
  If  $X\subseteq V- \{v\}$, $\abs{X}=2$, and
  $\rho_{G\setminus v}(X)\le 1$,
  then $\abs{X\cap A}\neq 1$.
\end{lemma}
\begin{proof}
  We obtain
  $\rho_{G\setminus v}(X\cap A)+\rho_G(X\cup A)\le \rho_{G\setminus
      v}(X)+\rho_G(A)\le \rho_G(A)+1$ by (\ref{eq:gga1}) of
  Lemma~\ref{lem:gga}.
  Suppose that $\abs{X\cap A}=1$.
  Since $A$ is fully closed,  $\rho_G(X\cup A)>\rho_G(A)$.
  Thus, $\rho_{G\setminus v}(X\cap A)=0$
  and therefore $\rho_G((X\cap A)\cup \{v\})\le 1$, 
  a contradiction to the
  assumption that $G$ is prime.
\end{proof}

\begin{lemma}\label{lem:pick}
  Let $G=(V,E)$ be a prime graph with $\abs{V}\ge 4$.
  Let $A$ be  a  fully closed set and let $v\in A$.
  If $G\setminus v$ is
  \rc{2}{1}{} but not prime, 
  then there is  a subset $X$ of $A- \{v\}$
  such that $\abs{X}=2$ and $\rho_{G\setminus v}(X)\le 1$.
\end{lemma}
\begin{proof}
  If $\abs{V}=4$, then it is trivial because
  every set $X$ of 2 vertices of $G\setminus v$ would
  satisfy $\rho_{G\setminus v}(X)\le 1$.
  Therefore we may assume that $\abs{V}\ge 5$.
  By Proposition~\ref{prop:rankconn}, $G\setminus v$ is connected.
  Since $G\setminus v$ is \rc{2}{1} 
  and $G\setminus v$ is not prime, 
  there exists a  subset $X$ of $V- \{v\}$ such that
  $\rho_{G\setminus  v}(X)\le 1$
  and $\abs{X}=2$.
  By Lemma~\ref{lem:triplet1side}, we may assume that
  $X\cap A=\emptyset$.
  Then $\abs{V-A}>2$ because $G$ is prime, $A\neq V$, and $A$ is fully closed. 

  Let  $B=V- A$.
  By \eqref{eq:gga2} of Lemma~\ref{lem:gga},
  \[\rho_{G\setminus v}(X)+\rho_G(B)\ge \rho_G(X)+\rho_{G\setminus v}(B).\]
  Since $G$ is prime, we have $\rho_G(X)=2$
  and therefore
  $2=\rho_G(A)\ge \rho_{G\setminus  v}(A-\{v\})+1$.
  Since $G\setminus v$ is \rc{2}{1} and $\abs{V-A}>2$, we have $\abs{A-\{v\}}\le 2$.
  As $A$ is fully closed, $\abs{A}=3$.

  Therefore, $\rho_{G\setminus v}(A- \{v\})\le 1$
  and $\abs{A- \{v\}}=2$.
\end{proof}

\begin{proof}[Proof of Proposition~\ref{prop:fully}]
  Suppose $A$ has no such vertex.
  Let $v_1\in A$.
  By Proposition~\ref{prop:bixby}, $G\setminus v_1$ or $G\pivot v_1w_1\setminus v_1$ for a neighbor $w_1$ of $v_1$ is \rc{2}{1}
  and therefore 
  we may assume that $G\setminus v_1$ is \rc{2}{1},
  because otherwise we may replace $G$ with $G\pivot v_1w_1$.
  Since $G\setminus v_1$ is not prime, there exists a set
  $X$ such that $\abs{X}=2$, $X\subseteq A-\{v_1\}$,
  and $\rho_{G\setminus v_1}(X)\le 1$ by Lemma~\ref{lem:pick}.
  Let $X=\{v_0,v_2\}$. Since $v_0$ has at least two neighbors, let $w_0$ be
  a neighbor of $v_0$ such that $w_0\neq v_1$.
  Then $G\setminus v_0$ or $G\pivot v_0w_0\setminus
    v_0$ is \rc{2}{1} by Proposition~\ref{prop:bixby}.
  If $G\setminus v_0$ is \rc{2}{1}, then a sequence
  $v_0$, $v_1$ in a graph $G$ has the desired property that
  both $G\setminus v_0$ and $G\setminus v_1$ are \rc{2}{1}.
  If $(G\pivot v_0w_0)\setminus v_0$ is \rc{2}{1}, then
  a sequence $v_0$, $v_1$ in a graph $G\pivot v_0w_0$ has the desired
  property.
  So we may assume that both $G\setminus v_0$ and $G\setminus v_1$ are
  \rc{2}{1}.
  By (ii) of Lemma~\ref{lem:intprime}, $v_2$ has a neighbor $w_2$ other than
  $v_0$ and $v_1$.
  Then $G\setminus v_2$ or $(G\pivot v_2w_2)\setminus v_2$ is
  \rc{2}{1} by Proposition~\ref{prop:bixby}.
  So we may assume that $G\setminus v_0$, $G\setminus v_1$, and
  $G\setminus v_2$ are \rc{2}{1}{} and
  $\rho_{G\setminus v_1}(\{v_0,v_2\})\le 1$.

  Therefore we can choose  a sequence $v_0, v_1,\ldots,v_k\in A$ of
  distinct vertices
  of a graph $G'$ pivot-equivalent to $G$
  such that
  $\rho_{G'\setminus v_i}(\{v_{i-1},v_{i+1}\})\le 1$ for all
  $1\le i<k$
  and $G'\setminus v_0$, $G'\setminus v_1$, $\ldots$, $G'\setminus v_k$
  are \rc{2}{1}{}
  and $k$ is maximum.
  By the above construction, $k\ge 2$.
  We may assume that $G=G'$.

  If $v_{i-1}$ is adjacent to only $v_i$ and $v_{i+1}$,  then
  $G\setminus v_{i-1}$ is prime by
  (ii) of Lemma~\ref{lem:intprime}.
  So $v_{i-1}$ has at least one neighbor other than $v_i$ and
  $v_{i+1}$.
  Similarly $v_{i+1}$ has at least one neighbor other than $v_i$ and
  $v_{i-1}$.
  As $\rho_{G\setminus v_i}(\{v_{i-1},v_{i+1}\})\le 1$, 
  we deduce that $v_{i-1}$ and $v_{i+1}$ are twins in $G\setminus v_i$.
  (Two vertices $x,y$ of $G$ are called \emph{twins} if no vertex in $V(G)-\{x,y\}$
  is adjacent to only one of $x$ and~$y$.)

  Since $G\setminus v_k$ is not prime, there is a set $X$
  such that $\rho_{G\setminus v_k}(X)\le 1$ and $\abs{X}=2$.
  By Lemma~\ref{lem:pick}, we may assume that $X\subseteq A$.

  We claim that $X\cap \{v_{k-2},v_{k-1}\}=\{v_{k-1}\}$.
  If   $X\cap\{v_{k-2},v_{k-1}\}=\{v_{k-2}\}$,
  then $\rho_{G\setminus v_{k-1}}(X\cup\{v_{k}\})\ge 2$
  because $G\setminus v_{k-1}$ is \rc{2}{1}.
  Then
  \[
    2\ge \rho_{G\setminus v_{k-1}} (\{v_{k-2},v_k\})
    +\rho_{G\setminus v_k} (X)
    \ge \rho_{G\setminus v_k}(\{v_{k-2}\})+
    \rho_{G\setminus v_{k-1}} (X\cup \{v_k\})\ge3;
  \]
  a contradiction.

  If $X\cap  \{v_{k-2},v_{k-1}\}=\emptyset$ or $X\cap
    \{v_{k-2},v_{k-1}\}=\{v_{k-2},v_{k-1}\}$, then
  let $Z=X$ or $Z=V(G\setminus v_k)- X$ so that
  $Z\cap \{v_{k-2},v_{k}\}=\{v_{k-2}\}$
  and $\rho_{G\setminus v_k}(Z)\le 1$. Then
  \[2\ge \rho_{G\setminus v_{k-1}}(\{v_{k-2},v_k\})+
    \rho_{G\setminus v_k}(Z)\ge
    \rho_{G\setminus v_k\setminus v_{k-1}}(\{v_{k-2}\})
    +\rho_G(Z\cup\{v_k\}).\]
  Since $G$ is prime
  and $v_{k-2}$ has a neighbor other than $v_{k-1}$ and $v_k$,
  we have $\rho_G(Z\cup\{v_k\})\ge 2$
  and $\rho_{G\setminus v_k\setminus v_{k-1}}(\{v_{k-2}\})\ge 1$;
  a contradiction.
  Thus we conclude that $X\cap \{v_{k-2},v_{k-1}\}=\{v_{k-1}\}$.

  Let $v_{k+1}$ be a vertex in $A$ such that $X=\{v_{k-1},v_{k+1}\}$.
  We claim that $v_{k+1}$ is distinct from $v_0$, $v_1$, $\ldots$,
  $v_k$.
  We already proved that $v_{k+1}\neq v_{k-1},v_k$.
  Suppose that $v_{k+1}=v_i$ for some $0\le i<k$ such that $i\equiv
    k+1\pmod 2$.
  Then for all $w\notin v_i,v_{i+1},\ldots,v_{k-1}$,
  $v_{i}$ and $v_{k-1}$ are twins in $G\setminus v_{i+1}\setminus
    v_{i+2}\cdots \setminus v_{k-2}$,
  because $v_{i+2j}$ and $v_{i+2j+2}$ are twins in $G\setminus
    v_{i+2j+1}$ for  $j=0,1,\ldots,(k-i-3)/2$.
  In particular if $w=v_k$, then we deduce that $v_k$ is adjacent to
  both $v_{k-1}$ and $v_i$
  or $v_k$ is nonadjacent to both $v_{k-1}$ and $v_i$.
  If $v_{i}$ has no neighbors other than $v_{k-1}$ and $v_k$, then
  by Lemma~\ref{lem:intprime}, $G\setminus v_i$ is prime,
  a contradiction.
  Thus $v_i$ has some neighbors other than $v_{k-1}$ and $v_k$.
  Similarly $v_{k-1}$ has some neighbors other than $v_{i}$ and $v_k$.
  Therefore we deduce that $v_i$ and $v_{k-1}$ are twins in
  $G\setminus v_k$. But if so, then $\rho_G(\{v_i,v_{k-1})\le1$, a
  contradiction to the assumption that $G$ is prime.

  Now suppose that $v_{k+1}=v_i$ for some $0\le i<k$ with $i\equiv
    k\pmod 2$. Then $v_0$, $v_1$, $\ldots$, $v_k$ have the same set of
  neighbors in $V(G)-\{v_0,v_1,\ldots,v_k\}$. This implies
  $\rho_G(\{v_0,v_1,\ldots,v_k\})\le 1$, a contradiction to the
  assumption that $G$ is prime.
  This completes the proof of the claim that $v_{k+1}$ is distinct
  from $v_0,v_1,\ldots,v_k$.

  Let $X=\{v_0,v_1,\ldots,v_k,v_{k+1}\}$.
  Since $G\setminus v_{k+1}$ is not prime, $v_{k+1}$ has
  a neighbor other than $v_k$ and $v_{k-1}$ by
  Lemma~\ref{lem:intprime}
  and therefore $v_{k+1}$ and $v_{k-1}$ are twins in $G\setminus v_k$.
  Thus we can argue that
  $v_0,v_2,\ldots$ have the same  set $N_0$ of neighbors in $V(G)-X$
  and $v_1,v_3,\ldots$ also have the same set $N_1$ of neighbors in
  $V(G)-X$.
  Since $\rho_G(X)\ge 2$, both  $N_0$ and $N_1$ are nonempty.
  This implies that $v_{k+1}$ has a neighbor $w$ in $V(G)-X$.

  We know that $G\setminus v_{k+1}$ is not \rc{2}{1},
  because $k$ is chosen as the maximum.

  Let $G'=G\pivot v_{k+1}w$. Then $G'\setminus v_{k+1}$ is
  \rc{2}{1}{}
  and it is easy to verify that $v_0,v_1,\ldots,v_{k+1}$ is a longer
  sequence satisfying all the conditions, which leads to a
  contradiction.
  Therefore
  there must exist a vertex $x\in A$ such that $G\setminus x$
  or $G/ x$ is prime.
\end{proof}

\section{Generating prime \rc{3}{3}{} pivot-minors}
\label{sec:main}

Here is our main theorem which provides
a proper prime \rc{3}{3}{} pivot-minor
of a \rc{3}{2}{} graph.

\begin{theorem}\label{thm:main}
  If   $G$ is a prime \rc{3}{2}{} graph
  with at least $10$ vertices, then
  $G$ has a prime \rc{3}{3}{} pivot-minor $H$
  with one fewer vertices.
\end{theorem}

We consider three following cases:
\begin{enumerate}
  \item $G$ is $3$-rank-connected,
  \item $G$ is prime \rc{3}{1} but not $3$-rank-connected,
  \item $G$ is prime \rc{3}{2}{} but not \rc{3}{1}.
\end{enumerate}

The cases (1) and (3) are easier than the case (2). We will treat
those two easy cases in this section.
\begin{lemma}\label{lem:4prime}
  Let $G$ be a $3$-rank-connected graph with at least $6$ vertices.
  For each vertex~$v$,
  $G\setminus v$ or $G/v$ is prime \rc{3}{2}{}.
\end{lemma}
\begin{proof}
  Let $u$ be a neighbor of $v$. 
  By Lemma~\ref{lem:pivotcutrank}, 
  both $G$ and $G\pivot uv$ are $3$-rank-connected
  and 
  by Proposition~\ref{prop:rankconn}, both $G\setminus v$ and $G/v=G\pivot uv \setminus v$ are prime.
  By Proposition~\ref{prop:bixby}, $G\setminus v$ or $G/v$ is \rc{3}{2}{}.
\end{proof}

\begin{lemma}\label{lemma:4prime4}
  Let $G$ be a prime \rc{3}{2}{}
  graph with  at least 10 vertices.
  If $G$ is not \rc{3}{1},
  then there exists a vertex $x$ of $G$
  such that
  $G\setminus x$ or $G/x$ is prime \rc{3}{3}.
\end{lemma}
\begin{proof}
  Since $G$ is prime but is not \rc{3}{1},
  there exists a subset  $X=\{a,b,c,d\}$  of $V(G)$ such that
  $\rho_G(\{a,b,c,d\})=2$.
  Since $\abs{V(G)}\ge 10$ and $G$ is \rc{3}{2},
  $\rho_G(X\cup \{v\})>2$ for all $v\in V(G)-X$.
  Therefore, $X$ is fully closed.
  By Proposition~\ref{prop:fully}, we may assume that
  $G\setminus a$ %
  is prime.

  Since $G\setminus a$ is prime, $\rho_{G\setminus a}(\{b,c,d\}) \ge  2$.
  Because $\rho_{G\setminus a}(\{b,c,d\})\le \rho_G(X)=2$,
  we have
  $\rho_{G\setminus a}(\{b,c,d\})=  2$.

  Let us consider a partition $(A,B)$ of $V(G\setminus a)$
  such that $\rho_{G\setminus a}(A)\le 2$.
  We may assume that $\abs{A\cap X}\ge \abs{B\cap X}$.

  (1) First suppose that $b,c,d\in A$.
  We claim that $\abs{A}\le 3$ or $\abs{B}\le 4$.
  By Lemma~\ref{lem:gga},
  $\rho_{G\setminus a}(A)+\rho_G(X)\ge \rho_{G\setminus a}(A\cap X)
    +\rho_G(A\cup X)$.
  Since
  $\rho_G(X)=2$ and $\rho_{G\setminus a}(\{b,c,d\})=2$,
  we have $\rho_G(A\cup \{a\})\le \rho_{G\setminus a}(A)$.
  Since $G$ is \rc{3}{2},
  we have $\abs{A}\le 3$ or $\abs{B}\le 4$.

  (2) Now consider the case when $\abs{A\cap X}=2$. We may assume that
  $A\cap X=\{b,c\}$.
  By the submodular inequality,
  $  \rho_{G\setminus a}(\{b,c,d\})+\rho_{G\setminus a}(A)
    \ge \rho_{G\setminus a}(A\cup\{d\})+
    \rho_{G\setminus a}(\{b,c\})$.
  Since  $G\setminus a$ is prime,
  we have $  \rho_{G\setminus a}(\{b,c\})=2$.
  Thus, we deduce that $\rho_{G\setminus a}(A\cup \{d\})\le
    \rho_{G\setminus a}(A)$.
  By applying (1) with $A\cup \{d\}$, we deduce that $\abs{A}+1\le 3$ or $\abs{B}-1\le 4$.

  By (1) and (2), $G\setminus a$ is \rc{3}{3}.
\end{proof}

\section{Triplets}
\label{sec:triplet}
To prove Theorem~\ref{thm:main},
we may assume that a graph $G$ is prime
\rc{3}{1} but not $3$-rank-connected
by Lemmas~\ref{lem:4prime} and~\ref{lemma:4prime4}.
Therefore, $G$ must have a set $T$ of vertices such that $\abs{T}=3$ and $\rho_G(T)=2$.
Our aim is to
prove that $G\setminus x$ is prime \rc{3}{3}{}
for some $x\in T$.
However, we can not achieve this
if $G\setminus x$ is not prime.
Therefore we wish to make sure that
$\rho_{G\setminus x}(T-\{x\})=2$ for some $x\in T$.

It turns out that we can achieve  a stronger statement;
there exists a graph $G'$ pivot-equivalent to $G$
such that
$\rho_{G'\setminus x}(T-\{x\})=2$ for \emph{all} $x\in T$.
In this section, we will prove this.

In a graph $G$, let us call a set $T$ of vertices a  \emph{triplet}
if
$\abs{T}=3$,
$\rho_G(T)=2$, and
$\rho_{G\setminus x}( T-\{x\})=2$
for each $x\in T$.

\begin{lemma}\label{lem:triplet}
  Let $G$ be  a prime graph and $\rho_G(\{a,b,c\})=2$.
  Then there exists a graph~$G'$ pivot-equivalent to $G$
  such that $\{a,b,c\}$ is a triplet of $G'$.
\end{lemma}
\begin{proof}
  A graphical proof is possible, but we present a proof using
  Lemma~\ref{lem:bixbyineq}.
  Suppose that $\{a,b,c\}$ is not a triplet.
  Let us assume that $\rho_{G\setminus c}(\{a,b\})=1$.
  Let $x$ be a neighbor of $c$, other than $a$ and $b$.
  Such $x$ should exist, because $\rho_G(\{a,b,c\})=2$.
  Then by Lemma~\ref{lem:bixbyineq},
  \[\rho_{G\pivot cx\setminus c}(\{a,b\})
    +\rho_{G\setminus c}(\{a,b\})\ge
    \rho_G(\{a,b\})+\rho_G(\{a,b,c\})-1= 3.\]
  So $\rho_{G\pivot cx\setminus c}(\{a,b\})=2$.
  Suppose that $\{a,b,c\}$ is not a triplet of $G'=G\pivot cx$.
  We may assume that $\rho_{G'\setminus b}(\{a,c\})=1$.
  We deduce that $\rho_{G\setminus b}(\{a,c\})=1$ and therefore
  $\rho_{G\setminus a}(\{b,c\})=2$, because if  a matrix with 3 rows
  has rank 2, then there are two linearly independent rows.
  Let $y$ be a neighbor of $b$ in $G'$, other than $a$ and $c$.
  Similarly to the above argument, we deduce that
  $\rho_{G'\pivot by\setminus b}(\{a,c\})=2$.
  In addition,
  $\rho_{G'\pivot by\setminus c}(\{a,b\})=\rho_{G'\setminus c}(\{a,b\})=2$
  and
  $\rho_{G'\pivot by\setminus a}(\{b,c\})=\rho_{G'\setminus
      a}(\{b,c\})
    =\rho_{G\setminus a}(\{b,c\})=2$.
\end{proof}
\begin{lemma}\label{lem:primetriplet}
  Let $G=(V,E)$ be a prime \rc{3}{1} graph with at least $8$ vertices.
  If $\{a,b,c\}$ is a triplet of $G$,
  then $G\setminus a$, $G\setminus b$, and $G\setminus c$ are prime.
\end{lemma}
\begin{proof}
  Suppose that $(X,Y)$ is a split of $G\setminus a$.
  Since $\rho_G(X\cup \{a\})\le 2$
  and $\rho_G(X)\le 2$,
  we have $\abs{X}\le 2$ or $\abs{Y}\le 2$.
  We may assume that $\abs{X}=2$.
  Since $X\neq \{b,c\}$, we may assume that $b\in Y$.

  If $c\in X$, then
  $\rho_{G\setminus b}(X\cup\{a\})
    +\rho_G(\{a,b,c\})
    \ge
    \rho_{G\setminus b}(\{a,c\})+\rho_G(X\cup\{a,b\})$ by Lemma~\ref{lem:gga}.
  Since $G$ is \rc{3}{1} and 
  $\min(\abs{X\cup \{a,b\}},\abs{Y-\{b\}})\ge 4$,
  we have 
  $\rho_{G}(X\cup\{a,b\})\ge 3$.
  Therefore
  $\rho_{G\setminus b}(X\cup\{a\})\ge 3$.
  We obtain a contradiction from Lemma~\ref{lem:gagb} as follows:
  \[
    \rho_{G\setminus a}(X)+\rho_{G\setminus b}(\{a,c\})=3
    \ge \rho_{G\setminus a}(\{c\})+\rho_{G\setminus b}(X\cup\{a\})\ge 4  .
  \]

  Now we assume that $c\notin X$.
  Then
  $\rho_{G\setminus a}(Y)+\rho_G(\{a,b,c\})
    \ge \rho_{G\setminus a}(\{b,c\})+\rho_G(Y\cup \{a\})$
  by Lemma~\ref{lem:gga}
  and therefore $\rho_{G}(Y\cup \{a\})\le 1$, contradictory to the
  assumption that $G$ is prime.
\end{proof}

\begin{lemma}\label{lem:sep}
  Let $G=(V,E)$ be a prime \rc{3}{1} graph.
  Let $\{a,b,c\}$ be a triplet of  $G$.
  Let $(X,Y)$ be a partition of $V(G\setminus a)$ such that
  $\rho_{G\setminus a}(X)=2$ and
  $\abs{X},\abs{Y}\ge 5$.
  Then $\abs{X\cap \{b,c\}}=1$
  and
  $\rho(X-\{b,c\},Y\cup\{a\})
    =\rho(X\cup\{a\},Y-\{b,c\})=3$.
\end{lemma}
\begin{proof}
  If $\{b,c\}\subseteq X$, then
  $\rho_{G\setminus a}(X)+\rho_{G}(\{a,b,c\})
    \ge \rho_{G\setminus a}(\{b,c\})+\rho_G(X\cup \{a\})
  $ by Lemma~\ref{lem:gga}
  and therefore
  $\rho_G(X\cup\{a\})\le 2$, contradictory to the assumption that
  $G$ is \rc{3}{1}.
  Therefore $\{b,c\}\not\subseteq X$.
  Similarly $\{b,c\}\not\subseteq Y$ by the symmetry between $X$ and $Y$.
  We conclude that $\abs{X\cap \{b,c\}}=1$. Then
  $\rho_G(X\cup\{a,b,c\})\ge 3$ because $G$ is \rc{3}{1}
  and $\abs{Y-\{b,c\}}\ge 4$.
  By symmetry, we may assume that $b\in X$ and $c\in Y$.
  Since $\rho(\{a,b,c\},Y-\{c\})=\rho(\{a,b\},Y-\{c\})$, 
  we have $\rho(X\cup\{a\},Y-\{c\})=\rho(X\cup\{a,c\},Y-\{c\})\ge 3$,
  and therefore
  $\rho(X\cup\{a\},Y- \{c\})=3$ as $\rho(X\cup \{a\},Y-\{c\})\le 
  \rho(X,Y-\{c\})+1\le \rho(X,Y)+1= 3$. 
  By symmetry between $X$ and $Y$,
  we have $\rho(X- \{b\},Y\cup\{a\})=3$.
\end{proof}

\section{Prime \rc{3}{1} graphs}
\label{sec:internally4prime}
To prove the remaining case (2) of Theorem~\ref{thm:main},
we will prove the following.
\begin{proposition}\label{prop:internally4prime}
  Let $G=(V,E)$ be a   prime \rc{3}{1} graph
  with at least $8$ vertices
  and a triplet  $\{a,b,c\}$.
  Then
  \begin{enumerate}[(i)]
    \item at least one of $G\setminus a$, $G\setminus b$, and $G\setminus c$
          is prime \rc{3}{2},
          or
    \item at least two of $G\setminus a$, $G\setminus b$, and $G\setminus
            c$    are prime \rc{3}{3}.
  \end{enumerate}
\end{proposition}
In the remaining of this section, we are going to prove Proposition~\ref{prop:internally4prime}.
We follow the outline of the proof of Hall~\cite[Theorem 3.1]{Hall2005} on internally $4$-connected matroids.
Let $G=(V,E)$ be a prime \rc{3}{1}{} graph with a triplet $\{a,b,c\}$ that does not satisfy (i).
By Lemma~\ref{lem:primetriplet}, all of $G\setminus a$, $G\setminus b$, and $G\setminus c$ are prime. 
So there exist partitions $(A_b,A_c)$, $(B_a,B_c)$, and $(C_a,C_b)$ of
$V(G\setminus a)$, $V(G\setminus b)$, $V(G\setminus c)$, respectively,
such that
\begin{enumerate}[(B1)]
  \item $\rho(A_b,A_c)=\rho(B_a,B_c)=\rho(C_a,C_b)=2$,
  \item $\abs{A_b}, \abs{A_c}, \abs{B_a}, \abs{B_c}, \abs{C_a}, \abs{C_b}\ge 5$,
  \item $b\in A_b$, $c\in A_c$, $a\in B_a$, $c\in B_c$, $a\in C_a$, and $b\in C_b$.
\end{enumerate}
Let us call $((A_b,A_c),(B_a,B_c),(C_a,C_b))$ a \emph{barrier} for a graph $G$ with a triplet $\{a,b,c\}$ 
if 
$(A_b,A_c)$, $(B_a,B_c)$, and $(C_a,C_b)$ are partitions of $V(G\setminus a)$, $V(G\setminus b)$, $V(G\setminus c)$, respectively 
satisfying (B1), (B2), and (B3).
\begin{lemma}\label{lem:sepequal}
  Let $((A_b,A_c),(B_a,B_c),(C_a,C_b))$ be a barrier of a prime \rc{3}{1} graph~$G$ with a triplet $\{a,b,c\}$. 
  Then the following hold.
  \begin{enumerate}[(i)]
    \item $\rho(A_b\cap B_c,A_c\cup B_a- \{a\})
            =\rho(A_b\cap B_c)$.
    \item \label{enum:sepequal2} $\rho(A_c\cap B_c,A_b\cup B_a-\{a,b\})
            =\rho(A_c\cap B_c)$.
    \item $\rho(B_a\cap A_c,A_b\cup B_c-\{b\})
            =\rho(B_a\cap A_c)$.
  \end{enumerate}
\end{lemma}
\begin{proof}
  Let $V=V(G)$. Then, 
  \begin{align*}
    \lefteqn{\rho(A_b\cap B_c,A_c\cup B_a- \{a\})}                                                         \\
     & =\rho(A_b\cap B_c,A_c\cup B_a)
     &                                          & \text{because }\rho(B_c,B_a- \{a\})=\rho(B_c,B_a),       \\
     & =\rho(A_b\cap B_c,A_c\cup B_a\cup\{b\})
     &                                          & \text{because }\rho(\{a,c\},V-\{a,b,c\})=\rho(\{a,b,c\},
    V-\{a,b,c\}),                                                                                         \\
    \lefteqn{\rho(A_c\cap B_c,A_b\cup B_a-\{a,b\})}.                                                       \\
     & =\rho(A_c\cap B_c,A_b\cup B_a-\{b\})
     &                                          & \text{because }\rho(B_c,B_a- a)=\rho(B_c,B_a),           \\
     & =\rho(A_c\cap B_c,A_b\cup B_a)
     &                                          & \text{because }\rho(A_c,A_b- b)=\rho(A_c,A_b),            \\
    \lefteqn{\rho(B_a\cap A_c,A_b\cup B_c-\{b\})}                                                          \\
     & =\rho(B_a\cap A_c,A_b\cup B_c)
     &                                          & \text{because }\rho(A_c,A_b- \{b\})=\rho(A_c,A_b),       \\
     & =\rho(B_a\cap A_c,A_b\cup B_c\cup \{a\})
     &                                          & \text{because }\rho(\{b,c\},V-\{a,b,c\})=\rho(\{a,b,c\},
    V-\{a,b,c\}).
    \qedhere
  \end{align*}
\end{proof}
\begin{lemma}\label{lem:sep2}
  Let $((A_b,A_c),(B_a,B_c),(C_a,C_b))$ be a barrier of a prime \rc{3}{1} graph~$G$ with a triplet $\{a,b,c\}$. Then the following hold.
  \begin{enumerate}[(i)]
    \item \label{enum:sep2-1}
          If $\abs{A_b\cap B_c}\ge 2$, then
          $\rho(B_a\cap A_c)\le 2$
          and $\abs{B_a\cap A_c}\le 3$.
    \item \label{enum:sep2-3}
          If $\rho(A_b\cap B_a,A_c\cup B_c)\ge 2$,
          then $\rho(A_c\cap B_c)\le 2$ and $\abs{A_c\cap B_c}\le 3$.
    \item \label{enum:sep2-4}
          If $\rho(A_b\cap B_a,A_c\cup B_c)\le 1$, then
          $\rho(A_b\cap B_a)\le 2$ and $\abs{A_b\cap B_a}\le 3$.
  \end{enumerate}
\end{lemma}
\begin{proof}
  \eqref{enum:sep2-1}
  Since $G$ is prime, $\rho(A_b\cap B_c)\ge 2$
  and by Lemma~\ref{lem:sepequal},
  $\rho(A_b\cap B_c,A_c\cup B_a-\{a\})\ge 2$.
  By the submodular inequality,
  \[
    \rho(A_b- \{b\},A_c)+
    \rho(B_c,B_a)
    \ge \rho(A_b\cap B_c,A_c\cup B_a)
    +\rho(A_b\cup B_c-\{b\}, B_a\cap A_c).\]
  Since $\rho(A_b-\{b\},A_c)\le \rho(A_b,A_c)= 2$ and $\rho(A_b\cap B_c,A_c\cup B_a)\ge
    \rho(A_b\cap B_c,A_c\cup B_a-\{a\})\ge 2$,
  we deduce that
  $\rho(B_a\cap A_c,A_b\cup B_c- \{b\})\le 2$.
  By Lemma~\ref{lem:sepequal},
  $\rho(B_a\cap A_c)\le 2$ and $\abs{B_a\cap A_c}\le 3$ because $G$ is
  \rc{3}{1}.

  \eqref{enum:sep2-3}
  By the submodular inequality,
  \[
    \rho(A_b- \{b\}, A_c)
    +\rho(B_a-\{a\},B_c)
    \ge \rho(A_b\cap B_a,A_c\cup B_c)
    +\rho(A_b\cup B_a-\{a,b\},A_c\cap B_c).\]
  Since $\rho(A_b,A_c)=\rho(B_a,B_c)=2$ and
  $\rho(A_c\cap B_c,A_b\cup B_a-\{a,b\})=\rho(A_c\cap B_c)$,
  we deduce that
  $\rho(A_c\cap B_c)\le 2$.
  Since $G$ is \rc{3}{1}, $\abs{A_c\cap B_c}\le 3$.

  \eqref{enum:sep2-4}
  From the assumption,
  $ \rho(A_b\cap B_a,A_c\cup B_c\cup \{a\})\le 2$.
  Then
  $
    \rho(A_b\cap B_a,A_c\cup B_c\cup\{a\})+
    \rho(V-\{a,b,c\},\{a,b,c\})
    \ge \rho(A_b\cap B_a,A_c\cup B_c\cup\{a,b\})
    +\rho(V- \{a,b,c\}, \{a,c\})$
  and we deduce that
  $\rho(A_b\cap B_a)\le 2$. Since $G$ is \rc{3}{1},
  $\abs{A_b\cap B_a}\le 3$.
\end{proof}

\begin{lemma}\label{lem:abba}
  Let $((A_b,A_c),(B_a,B_c),(C_a,C_b))$ be a barrier of a prime \rc{3}{1} graph~$G$ with a triplet $\{a,b,c\}$. 
  \begin{enumerate}[(i)]
    \item 
    If $\abs{A_b\cap B_a}>2$ and $\abs{A_c\cap B_c}>3$,
    then
    $G\setminus c$ is prime \rc{3}{3}.
    \item If $\abs{A_c\cap C_a}>2$ and $\abs{A_b\cap C_b}>3$, 
    then
    $G\setminus b$ is prime \rc{3}{3}.
    \item If $\abs{B_c\cap C_b}>2$ and $\abs{B_a\cap C_a}>3$,
    then
    $G\setminus a$ is prime \rc{3}{3}.
  \end{enumerate}
\end{lemma}
\begin{proof}
  By symmetry, it is enough to prove (i).
  By Lemma~\ref{lem:sep2}\eqref{enum:sep2-3}, $\rho(A_b\cap B_a,A_c\cup B_c)\le 1$.
  By Lemma~\ref{lem:sep2}\eqref{enum:sep2-4}, we have 
  $\rho(A_b\cap B_a)\le 2$ and $\abs{A_b\cap B_a}=3$. Since $\abs{A_b\cap B_a}=3$, we have $\rho(A_b\cap
    B_a)=2$.
  By Lemma~\ref{lem:primetriplet}, $G\setminus c$ is prime.
  To show that $G\setminus c$ is \rc{3}{3},
  it is enough to show that
  $\abs{C_a}\le 5$ or $\abs{C_b}\le 5$
  because the statement holds for all barriers $((A_b,A_c),(B_a,B_c),(C_a',C_b'))$ sharing $(A_b,A_c)$ and $(B_a,B_c)$.
  By the symmetry between $a$ and $b$,
  we have two cases to check:
  $\abs{(A_b\cap B_a)\cap C_a}=3$
  or $\abs{(A_b\cap B_a)\cap C_a}=2$.

  To check the first case, let us assume that $A_b\cap B_a\subseteq
    C_a$.
  From the submodular inequality,
  \[
    \rho((A_b\cap B_a)\cup\{b\}, A_c\cup B_c)
    +\rho(C_a,C_b)
    \ge \rho(C_a\cup \{b\},C_b- \{b\})
    + \rho(A_b\cap B_a,A_c\cup B_c\cup \{b\}).
  \]
  We know that $\rho(C_a,C_b)=2$
  and
  $\rho((A_b\cap B_a)\cup\{b\}, A_c\cup B_c)\le 1+\rho(A_b\cap
    B_a,A_c\cup B_c)\le 2$.
  Since
  $\rho(\{b,c\},V-\{a,b,c\})=\rho(\{a,b,c\},V-\{a,b,c\})$,
  we have
  $\rho(A_b\cap B_a,A_c\cup B_c\cup \{b\})=\rho(A_b\cap B_a)$.
  Since $G$ is prime, $\rho(A_b\cap B_a)\ge 2$.
  So we deduce that
  $\rho(C_a\cup \{b\},C_b-\{b\})\le 2$.
  Since
  $\rho(\{a,b\},V-\{a,b,c\})=\rho(\{a,b,c\},V-\{a,b,c\})$,
  we deduce that
  $\rho(C_a\cup\{b,c\},C_b-
    \{b\})=\rho(C_a\cup\{b\},C_b-\{b\})\le 2$.
  Since $G$ is \rc{3}{1}, we have
  $\abs{C_b- \{b\}}\le 3$. So $\abs{C_b}\le 4$.

  Now we need to check the remaining case in which $\abs{(A_b\cap B_a)\cap
      C_a}=2$.
  Let $x$ be the unique element in $A_b\cap B_a\cap C_b$.
  Let $A_b\cap B_a\cap C_a=\{y,z\}$.
  Since $x\in B_a$, $x\neq b$.
  From the submodular inequality,
  \[
    \rho(C_a,C_b)+\rho(\{x,y,z\},V-\{x,y,z\})
    \ge \rho(\{y,z\},V-\{y,z\})+
    \rho(C_a\cup\{x\}, C_b-\{x\}),
  \]
  and we deduce that $\rho(C_a\cup\{x\},C_b-\{x\})\le 2$.
  Again we use the submodular inequality, and obtain that
  \begin{multline*}
    \rho(\{x,y,z,b\},V-\{x,y,z,b\})+
    \rho(C_a\cup\{x\},C_b-\{x\})\\
    \ge
    \rho(C_a\cup\{x,b\},C_b- \{b,x\})
    +\rho(\{x,y,z\},V- \{x,y,z\}).
  \end{multline*}
  We have $\rho(\{x,y,z,b\})\le 1+\rho(A_b\cap B_a,A_c\cup B_c)=2$
  because $b\notin A_c\cup B_c$.
  So it follows that
  $\rho(C_a\cup \{x,b\},C_b-\{b,x\})\le 2$.
  Since
  $\rho(\{a,b\},V-\{a,b,c\})=\rho(\{a,b,c\},V-\{a,b,c\})$,
  we have
  $\rho(C_a\cup \{x,b,c\},C_b-\{b,x\})=
    \rho(C_a\cup \{x,b\},C_b-\{b,x\})\le 2$.
  Since $G$ is \rc{3}{1} and $\rho(C_b-\{b,x\})\le 2$,
  we have $\abs{C_b}-2\le 3$.
\end{proof}

\begin{lemma}\label{lem:meatiest}
  Let $((A_b,A_c),(B_a,B_c),(C_a,C_b))$ be a barrier of a prime \rc{3}{1} graph~$G$ with a triplet $\{a,b,c\}$
  such that 
  each of $\min(\abs{A_b},\abs{A_c})$, $\min(\abs{B_a},\abs{B_c})$, and $\min(\abs{C_a},\abs{C_b})$ is
  maximum.
  Then at least two of $G\setminus a$, $G\setminus b$, and $G\setminus c$
  are  \rc{3}{3}
  if at least one of the following inequalities holds.
  \begin{align*}
    \abs{A_b\cap B_a}&\le 1,  
    & \abs{B_c\cap C_b}&\le 1,
    & \abs{A_c\cap C_a}&\le 1, \\
    \abs{A_b\cap B_c}&\le 1,
    &\abs{B_c\cap C_a}&\le 1,
    & \abs{C_a\cap A_b}&\le 1, \\
    \abs{C_b\cap B_a}&\le 1,
    &\abs{B_a\cap A_c}&\le 1,
    & \abs{A_c\cap C_b}&\le 1, \\
    \abs{A_c\cap B_c}&\le 2,
    & \abs{B_a\cap C_a}&\le 2,
    & \abs{A_b\cap C_b}&\le 2.
  \end{align*}
\end{lemma}
\begin{proof}
  (1) Suppose that $\abs{A_b\cap B_a}\le 1$.
  If $\abs{B_a\cap A_c}>3$, then
  by Lemma~\ref{lem:sep2}\eqref{enum:sep2-1}, $\abs{A_b\cap B_c}\le 1$.
  Then $\abs{A_b}=1+\abs{A_b\cap B_a}+\abs{A_b\cap B_c}\le 3$, a contradiction.
  So $\abs{B_a\cap A_c}\le 3$. Because $\abs{B_a}\ge 5$,
  we have $\abs{B_a\cap A_c}=3$ and $\abs{B_a}=5$. By Lemma~\ref{lem:sep2}\eqref{enum:sep2-1} and symmetry between $a$ and $b$, 
  $\abs{A_b\cap B_c}\le 3$ and therefore $\abs{A_b}= 5$.
  Therefore, $G\setminus a$ and $G\setminus b$
  are  \rc{3}{3}
  by the assumption that each of $\min(\abs{A_b},\abs{A_c})$ and $\min(B_a,B_c)$ is maximum.

  By symmetry, the conclusion holds if $\abs{B_c\cap C_b}\le 1$ or $\abs{A_c\cap C_a}\le 1$.

  (2) Suppose that $\abs{B_a\cap A_c}\le 1$.
  Since $\abs{A_c}\ge 5$, we have $\abs{A_c\cap B_c}\ge 4$, and therefore
  $\rho(A_b\cap B_a,A_c\cup B_c)\le 1$   by Lemma~\ref{lem:sep2}\eqref{enum:sep2-3}
  and
  $\abs{A_b\cap B_a}\le 3$
  by Lemma~\ref{lem:sep2}\eqref{enum:sep2-4}.
  Then $\abs{B_a}=1+\abs{A_b\cap B_a}+\abs{B_a\cap A_c}\le 5$.
  So $G\setminus b$ is \rc{3}{3}
  by the assumption that $\min(B_a,B_c)$ is maximum.
  Because the equality holds, $\abs{A_b\cap B_a}=3$.
  By Lemma~\ref{lem:abba}, $G\setminus c$ is
  \rc{3}{3}
  by the assumption that $\min(C_a,C_b)$ is maximum.

  By symmetry, the conclusion holds if $\abs{A_b\cap B_c}\le 1$, $\abs{B_c\cap C_a}\le 1$,
  $\abs{C_a\cap A_b}\le 1$, 
  $\abs{C_b\cap B_a}\le 1$,
    or $\abs{A_c\cap C_b}\le 1$.

  (3) Suppose that $\abs{A_c\cap B_c}\le 2$.
  Since $\abs{B_c}\ge 5$, we have $\abs{A_b\cap B_c}\ge 3$.
  By Lemma~\ref{lem:sep2}\eqref{enum:sep2-1}, $\abs{B_a\cap A_c}\le 3$.
  Then $\abs{A_c}=\abs{B_a\cap A_c}+\abs{A_c\cap B_c}\le 5$.
  Since the equality holds, $\abs{B_a\cap A_c}=3$.
  By Lemma~\ref{lem:sep2}\eqref{enum:sep2-1} and the symmetry between $b$ and $c$, we have $\abs{A_b\cap B_c}\le 3$.
  So $\abs{B_c}=\abs{A_b\cap B_c}+\abs{A_c\cap B_c}\le 5$.
  So $G\setminus a$ and $G\setminus b$ are
  \rc{3}{3}
  by the assumption that each of $\min(\abs{A_b},\abs{A_c})$ and $\min(B_a,B_c)$ is maximum.

  By symmetry, the conclusion holds if  $\abs{B_a\cap C_a}\le 2$ or 
  $\abs{A_b\cap C_b}\le 2$.
\end{proof}

\begin{lemma}\label{lem:sum1}
  Let $((A_b,A_c),(B_a,B_c),(C_a,C_b))$ be a barrier of a prime \rc{3}{1} graph~$G$ with a triplet $\{a,b,c\}$
  such that 
  each of $\min(\abs{A_b},\abs{A_c})$, $\min(\abs{B_a},\abs{B_c})$, and $\min(\abs{C_a},\abs{C_b})$ is
  maximum.
  If 
  $\abs{A_b\cap B_c\cap C_a}+\abs{A_c\cap C_b\cap B_a}\le 1$,
  then 
  at least two of 
  $G\setminus a$, $G\setminus b$, and $G\setminus
    c$ are \rc{3}{3}.
\end{lemma}
\begin{proof}
  By the symmetry between $b$ and $c$,
  we may assume that 
  $\abs{A_b\cap B_c\cap C_a}=0$ and $\abs{A_c\cap C_b\cap B_a}\le 1$.
  By Lemma~\ref{lem:meatiest}, we may assume that 
  $\abs{A_b\cap B_c}> 1$,
  $\abs{B_c\cap C_a}> 1$,
  $\abs{C_a\cap A_b}>1$, 
  $\abs{C_b\cap B_a}>1$,
  $\abs{B_a\cap A_c}>1$,
  and 
  $\abs{A_c\cap C_b}>1$.
  
  Observe that
  \begin{align*}
    \abs{A_b\cap B_c\cap C_b}&=\abs{A_b\cap B_c}-\abs{A_b\cap B_c\cap C_a}\ge 2-0=2,\\
    \abs{A_c\cap B_c\cap C_b}&=\abs{A_c\cap C_b}-\abs{A_c\cap C_b\cap B_a}\ge 2-1=1,\\
    \abs{A_b\cap B_a\cap C_a}&=\abs{C_a\cap A_b}-\abs{A_b\cap B_c\cap C_a} \ge 2-0=2,\\
    \abs{A_c\cap B_a\cap C_a}&=\abs{B_a\cap A_c}-\abs{A_c\cap C_b\cap B_a} \ge 2-1=1. 
  \end{align*}
  Therefore, $\abs{B_c\cap C_b}=\abs{A_b\cap B_c\cap C_b}+\abs{A_c\cap B_c\cap C_b}\ge 3$ and 
  $\abs{B_c\cap C_a}=\abs{A_b\cap B_a\cap C_a}+\abs{A_c\cap B_a\cap C_a}+1\ge 4$. By Lemma~\ref{lem:abba}, $G\setminus a$ is \rc{3}{3}.

  By symmetry, $\abs{A_b\cap B_a}\ge 3$, $\abs{A_c\cap C_a}\ge 3$, $\abs{A_c\cap B_c}\ge 4$, and $\abs{A_b\cap C_b}\ge 4$. From Lemma~\ref{lem:abba}, we deduce that $G\setminus b$ and $G\setminus c$ are \rc{3}{3}.
\end{proof}

\begin{lemma}\label{lem:11}
  Let $((A_b,A_c),(B_a,B_c),(C_a,C_b))$ be a barrier of a prime \rc{3}{1} graph~$G$ with a triplet $\{a,b,c\}$
  such that 
  each of $\min(\abs{A_b},\abs{A_c})$, $\min(\abs{B_a},\abs{B_c})$, and $\min(\abs{C_a},\abs{C_b})$ is
  maximum.
  If $\abs{A_b\cap B_c\cap C_a}=\abs{A_c\cap C_b\cap B_a}=1$, then
  at least two of   $G\setminus a$, $G\setminus b$, and $G\setminus
    c$ are \rc{3}{3}.
\end{lemma}
\begin{proof}
  Since all graphs up to $11$ vertices are \rc{3}{3}, we may assume that $\abs{V}\ge 13$.
  By Lemma~\ref{lem:meatiest}, we may assume that 
  $\abs{A_b\cap B_a}>1$, 
  $\abs{A_c\cap C_a}>1$, 
  $\abs{B_c\cap C_b}>1$, 
  $\abs{A_b\cap B_c}> 1$,
  $\abs{B_c\cap C_a}> 1$,
  $\abs{C_a\cap A_b}>1$, 
  $\abs{C_b\cap B_a}>1$,
  $\abs{B_a\cap A_c}>1$,
  $\abs{A_c\cap C_b}>1$, 
  $\abs{A_c\cap B_c}>2$, 
  $\abs{A_b\cap C_b}>2$, 
  and 
  $\abs{B_a\cap C_a}>2$.
  By Lemma~\ref{lem:sep2}\eqref{enum:sep2-1} and symmetry in $\{a,b,c\}$, we have 
  $\abs{A_b\cap B_c}, 
  \abs{B_c\cap C_a},
  \abs{C_a\cap A_b},
  \abs{C_b\cap B_a},
  \abs{B_a\cap A_c},
  \abs{A_c\cap C_b}\in\{2,3\}$
  and 
  $\rho(A_b\cap B_c), 
  \rho(B_c\cap C_a),
  \rho(C_a\cap A_b),
  \rho(C_b\cap B_a),
  \rho(B_a\cap A_c),
  \rho(A_c\cap C_b)\le 2$.
  
  First, observe that 
  \begin{align*}
    \lefteqn{\abs{B_a\cap A_c}+\abs{C_b\cap B_a}+\abs{A_c\cap C_b}
    +\abs{A_b\cap B_c}+\abs{B_c\cap C_a}+\abs{C_a\cap A_b}}\\
    &= \abs{ (B_a\cap A_c)\cup (C_b\cap B_a)\cup (A_c\cap C_b)} + 2\abs{A_c\cap B_a\cap C_b}\\
    &\quad + \abs{(A_b\cap B_c)\cup (B_c\cap C_a)\cup (C_a\cap A_b)}
    +2\abs{A_b\cap B_c\cap C_a}\\
    &= \abs{V}-\abs{\{a,b,c\}}+2+2\ge 14.
  \end{align*}
  We may assume that 
  \(\abs{B_a\cap A_c}+\abs{C_b\cap B_a}+\abs{A_c\cap C_b}\ge 7\)
  by swapping $b$ and $c$ if necessary.
  By exchanging $(a,b,c)$ with $(b,c,a)$ or $(c,a,b)$ if necessary, we may assume that 
  $\abs{B_a\cap A_c}= \max(\abs{B_a\cap A_c}, \abs{C_b\cap B_a}, \abs{A_c\cap C_b})$
  and therefore \[ \abs{B_a\cap A_c}=3\] 
  and $\abs{B_a\cap A_c\cap C_a}= \abs{B_a\cap A_c}-\abs{B_a\cap A_c\cap C_b} = 2$.
  Since $G$ is prime, $\rho(A_c\cap B_a\cap C_b)=2$.

  Since 
  \begin{align*}
    \abs{(B_a\cap A_c)\cup (A_c\cap C_a)}%
    &=\abs{A_c\cap C_a}+ \abs{B_a\cap A_c\cap C_b}=4, \\
    \abs{B_a\cap A_c\cap C_a}&=\abs{B_a\cap A_c}-\abs{B_a\cap A_c\cap C_b}\ge 2,
  \end{align*}
  we have 
  $\rho((B_a\cap A_c)\cup (A_c\cap C_a))\ge 3$ 
  and $\rho(B_a\cap A_c\cap C_a)\ge 2$.
  Previously we assumed that $\rho(B_a\cap A_c)\le 2$.
  By the submodular inequality,
  \[\rho(B_a\cap A_c)+\rho(A_c\cap C_a)\ge \rho(B_a\cap A_c\cap C_a)+
    \rho((B_a\cap A_c)\cup (A_c\cap C_a))\ge 5\]
  and therefore $\rho(A_c\cap C_a)>2$. By \eqref{enum:sep2-3} and \eqref{enum:sep2-4} of Lemma~\ref{lem:sep2} with the symmetry between $b$ and $c$, 
  we deduce that $\abs{A_b\cap C_b}\le 3$.
  By Lemma~\ref{lem:meatiest}, we may assume that \[ \abs{A_b\cap C_b}=3.\] 
  Observe that 
  $\abs{A_b\cap C_b}=1+\abs{A_b\cap C_b\cap B_a}+\abs{A_b\cap B_c\cap C_b}
  =1+\abs{C_b\cap B_a}-\abs{A_c\cap C_b\cap B_a}
  +\abs{A_b\cap B_c}-\abs{A_b\cap B_c\cap C_a}
  =\abs{C_b\cap B_a}+\abs{A_b\cap B_c}-1\ge 2+2-1$
  and therefore 
  \[ 
  \abs{A_b\cap B_c}=\abs{C_b\cap B_a}=2.
  \] 

  Since $\abs{B_a\cap C_a}= 1+\abs{B_a\cap A_c\cap C_a}+\abs{B_a\cap A_c\cap C_b}\ge 4$, 
  from \eqref{enum:sep2-3} and \eqref{enum:sep2-4} of Lemma~\ref{lem:sep2} and the symmetry between $a$ and $c$,
  we deduce that 
  \(\rho(B_c\cap C_b)\le 2\).
  Previously we assumed that $\rho(A_c\cap C_b)\le 2$.
  If $\abs{A_c\cap C_b}=3$, then $\abs{A_c\cap B_c\cap C_b}=\abs{A_c\cap C_b}-\abs{A_c\cap C_b\cap B_a}=2$, 
  $\rho(A_c\cap B_c\cap C_b)\ge 2$, and 
  \begin{align*}
    \abs{(A_c\cap C_b)\cup (B_c\cap C_b)}
    &=3+\abs{A_b\cap B_c\cap C_b}\\
    &\ge 3+\abs{A_b\cap B_c}-\abs{A_b\cap B_c\cap C_a}\\
    &\ge 3+2 -1 = 4.
  \end{align*}
  By the submodular inequality,
  \[
    2+2\ge \rho(A_c\cap C_b)+\rho(B_c\cap C_b)\ge
    \rho(A_c\cap B_c\cap C_b)+\rho((A_c\cap C_b)\cup (B_c\cap C_b))\]
  and therefore $\rho((A_c\cap C_b)\cup (B_c\cap C_b))\le 2$, 
  contradicting to the assumption that $G$ is \rc{3}{1}.
  Therefore, we conclude that \[ \abs{A_c\cap C_b}=2.\]

  Then $\abs{C_b}=\abs{A_b\cap C_b}+\abs{A_c\cap C_b}=3+2$ and therefore $G\setminus
    c$ is \rc{3}{3}, because $(C_a,C_b)$ was
  chosen to maximize $\min(\abs{C_a},\abs{C_b})$.

  Suppose that $\abs{B_c\cap C_a}=\abs{C_a\cap A_b}=3$.
  Since $\abs{A_c\cap B_c}=1+\abs{A_c\cap B_c\cap C_a}+\abs{A_c\cap B_c\cap C_b}
  =1+\abs{B_c\cap C_a}-\abs{A_b\cap B_c\cap C_a}+\abs{A_c\cap C_b}-\abs{A_c\cap C_b\cap B_a}\ge 1+3-1+2-1=4$, we have
  $\rho(A_b\cap B_a)\le 2$ by
  \eqref{enum:sep2-3} and \eqref{enum:sep2-4} of Lemma~\ref{lem:sep2}.
  By the submodular
  inequality,
  \[
    2+2\ge \rho(A_b\cap B_a)+\rho(C_a\cap A_b)\ge
    \rho(C_a\cap A_b\cap B_a)+
    \rho((A_b\cap B_a)\cup (C_a\cap A_b)).
  \]
  Since $\abs{C_a\cap A_b\cap B_a}=\abs{C_a\cap A_b}-\abs{C_a\cap A_b\cap B_c}=2$, we have $  \rho(C_a\cap A_b\cap B_a)\ge 2$
  and therefore $\rho((A_b\cap B_a)\cup (C_a\cap A_b))\le 2$.
  This contradicts to the assumption that 
  $G$ is \rc{3}{1} 
  because 
  $\abs{(A_b\cap B_a)\cup (C_a\cap A_b)}=\abs{C_a\cap A_b}+\abs{A_b\cap B_a\cap C_b}=3+\abs{C_b\cap B_a}-\abs{C_b\cap B_a\cap A_c}\ge 3+2-1$.
  Therefore  $\abs{B_c\cap C_a}=2$ or $\abs{C_a\cap A_b}=2$.

  If $\abs{C_a\cap A_b}=2$, then
  $\abs{A_b}=\abs{A_b\cap C_b}+\abs{C_a\cap A_b}=3+2$ and so $G\setminus a$
  is \rc{3}{3}.
  If $\abs{B_c\cap C_a}=2$, then
  \begin{align*}
    \abs{B_c}&=1+\abs{B_c\cap C_a}+\abs{B_c\cap C_b}\\
    &= 3+ \abs{A_b\cap B_c\cap C_b}+\abs{A_c\cap B_c\cap C_b}\\
    &= 3+(\abs{A_b\cap B_c}-\abs{A_b\cap B_c\cap C_a})
    +(\abs{A_c\cap C_b}-\abs{A_c\cap C_b\cap B_a}) \\
    &= 3+2-1+2-1=5
  \end{align*}
 and so $G\setminus b$
  is \rc{3}{3}.
\end{proof}
\begin{lemma}\label{lem:lambda2}
  Let $((A_b,A_c),(B_a,B_c),(C_a,C_b))$ be a barrier of a prime \rc{3}{1} graph~$G$ with a triplet $\{a,b,c\}$
  such that 
  each of $\min(\abs{A_b},\abs{A_c})$, $\min(\abs{B_a},\abs{B_c})$, and $\min(\abs{C_a},\abs{C_b})$ is
  maximum.
    If $\abs{A_b\cap B_c\cap C_a}\ge 2$ and 
    $\abs{(A_b\cap B_c)\cup (B_c\cap C_a)\cup (C_a\cap A_b)}>3$, 
    then  at least two of   $G\setminus a$, $G\setminus b$, and $G\setminus
    c$ are \rc{3}{3}.

\end{lemma}
\begin{proof}
  Since all graphs up to $11$ vertices are \rc{3}{3}, we may assume that $\abs{V}\ge 13$.
  By Lemma~\ref{lem:meatiest}, we may assume that 
  $\abs{A_b\cap B_c}> 1$,
  $\abs{B_c\cap C_a}> 1$,
  $\abs{C_a\cap A_b}>1$, 
  $\abs{C_b\cap B_a}>1$,
  $\abs{B_a\cap A_c}>1$, and 
  $\abs{A_c\cap C_b}>1$.
  By Lemma~\ref{lem:sep2}\eqref{enum:sep2-1} and symmetry in $\{a,b,c\}$, we have 
  $\abs{A_b\cap B_c}, 
  \abs{B_c\cap C_a},
  \abs{C_a\cap A_b},
  \abs{C_b\cap B_a},
  \abs{B_a\cap A_c},
  \abs{A_c\cap C_b}\in\{2,3\}$
  and 
  $\rho(A_b\cap B_c), 
  \rho(B_c\cap C_a),
  \rho(C_a\cap A_b),
  \rho(C_b\cap B_a),
  \rho(B_a\cap A_c),
  \rho(A_c\cap C_b)\le 2$.

  We may assume that $\abs{A_b\cap B_c\cap C_a}=2$, because otherwise $\abs{A_b\cap B_c}=\abs{B_c\cap C_a}=\abs{C_a\cap A_b}=\abs{A_b\cap B_c\cap C_a}=3$.
  By the submodular inequality,
  \[
    2+2\ge \rho(A_b\cap B_c)+\rho(B_c\cap C_a)
    \ge \rho(A_b\cap B_c\cap C_a)+\rho((A_b\cap B_c)\cup (B_c\cap
    C_a)).
  \]
  Since $G$ is prime, we have 
  $\rho(A_b\cap B_c\cap C_a)\ge 2$
  and therefore, $\rho((A_b\cap B_c)\cup (B_c\cap C_a))\le 2$.
  Since $G$ is \rc{3}{1},
  we have $\abs{(A_b\cap B_c)\cup (B_c\cap C_a)}\le 3$.
  Thus,
  \[ \abs{A_b\cap B_c}+\abs{B_c\cap C_a}\le 3+\abs{A_b\cap B_c\cap C_a}=5.\]
  Thus at least one of $A_b\cap B_c$ and $B_c\cap C_a$ is equal to $A_b\cap B_c\cap C_a$. 
  Similarly at least one of $A_b\cap B_c$ and $C_a\cap A_b$ is equal to $A_b\cap B_c\cap C_a$
  and 
  at least one of $B_c\cap C_a$ and $C_a\cap A_b$ is equal to $A_b\cap B_c\cap C_a$.
  Thus, at least two of $A_b\cap B_c$, $B_c\cap C_a$, and $C_a\cap A_b$ are equal to $A_b\cap B_c\cap C_a$ and so 
  $\abs{(A_b\cap B_c)\cup (B_c\cap C_a)\cup (C_a\cap A_b)}\le 3$, contradicting the assumption. 
\end{proof}
\begin{lemma}\label{lem:lambda2'}
  Let $((A_b,A_c),(B_a,B_c),(C_a,C_b))$ be a barrier of a prime \rc{3}{1} graph~$G$ with a triplet $\{a,b,c\}$
  such that 
  each of $\min(\abs{A_b},\abs{A_c})$, $\min(\abs{B_a},\abs{B_c})$, and $\min(\abs{C_a},\abs{C_b})$ is
  maximum.
  If $\abs{A_b\cap B_c\cap C_a}\ge 2$ and $\abs{A_c\cap C_b\cap B_a}\ge 2$, then 
  at least two of   $G\setminus a$, $G\setminus b$, and $G\setminus c$ are \rc{3}{3}.
\end{lemma}
\begin{proof}
  Since all graphs up to $11$ vertices are \rc{3}{3}, we may assume that $\abs{V}\ge 13$.
  Since 
  \[ 
  (A_b\cap B_c)\cup (B_c\cap C_a)\cup (C_a\cap A_b)\cup 
  (A_c\cap C_b)\cup (C_b\cap B_a)\cup (B_a\cap A_c)=V-\{a,b,c\},\] 
  we deduce that 
  $\abs{(A_b\cap B_c)\cup (B_c\cap C_a)\cup (C_a\cap A_b)}\ge 5$
  or 
  $\abs{(A_c\cap C_b)\cup (C_b\cap B_a)\cup (B_a\cap A_c)}\ge 5$.
  The conclusion follows by Lemma~\ref{lem:lambda2}.
\end{proof}

\begin{lemma}\label{lem:commonextreme1}
  Let $((A_b,A_c),(B_a,B_c),(C_a,C_b))$ be a barrier of a prime \rc{3}{1} graph~$G$ with a triplet $\{a,b,c\}$
  such that 
  each of $\min(\abs{A_b},\abs{A_c})$, $\min(\abs{B_a},\abs{B_c})$, and $\min(\abs{C_a},\abs{C_b})$ is
  maximum.
    If $A_b\cap B_a\cap C_a=A_b\cap B_c\cap C_b=\emptyset$,
    $\rho(A_b\cap B_a,A_c\cup B_c)\le 1$, 
    and $\abs{C_b\cap B_a}=3$, then 
    at least two of   $G\setminus a$, $G\setminus b$, and $G\setminus c$ are \rc{3}{3}.

\end{lemma}

\begin{proof}
  Suppose that $\abs{B_c\cap C_a}>1$. Then 
  by Lemma~\ref{lem:sep2}\eqref{enum:sep2-1} and symmetry in $\{a,b,c\}$, we have 
  $
  \rho(C_b\cap B_a)
  \le 2$.

  Since   $\rho(A_b\cap B_a,A_c\cup B_c)\le 1$, 
  we have 
  $\rho((A_b\cap B_a)\cup\{b\},A_c\cup B_c)\le 2$.
  As $A_b\cap B_a\cap C_a=A_b\cap B_c\cap C_b=\emptyset$, we have 
  $A_b\cap B_a=A_b\cap B_a\cap C_b=A_b\cap C_b-\{b\}$.
  Thus $\rho(A_b\cap C_b,A_c\cup C_a-\{a\})\le 2$.
  By Lemma~\ref{lem:sepequal}\eqref{enum:sepequal2}
  and the symmetry between $b$ and $c$,
  \[ 
    \rho(A_b\cap C_b)=\rho(A_b\cap C_b,A_c\cup C_a-\{a,c\})
    \le \rho(A_b\cap C_b,A_c\cup C_a-\{a\})
    \le 2.
  \]
  By the submodular inequality,
  \[  2+2\ge \rho(A_b\cap C_b)+\rho(C_b\cap B_a)
    \ge \rho(A_b\cap B_a\cap C_b)+\rho((A_b\cap C_b)\cup (C_b\cap B_a)).
  \]
  Since $G$ is prime and $\abs{A_b\cap B_a\cap C_b}=\abs{A_b\cap B_a}-\abs{A_b\cap B_a\cap C_a}\ge 2$,
  we have $\rho(A_b\cap B_a\cap C_b)\ge 2$.
  Therefore, $\rho((A_b\cap C_b)\cup (C_b\cap B_a))\le 2$, 
  which implies that 
  $\abs{(A_b\cap C_b)\cup (C_b\cap B_a)}\le 3$,
  because $G$ is \rc{3}{1}.
  However, $\abs{(A_b\cap C_b)\cup (C_b\cap B_a)}
  = \abs{A_b\cap B_c\cap C_b}+1+\abs{C_b\cap B_a}= 1+\abs{C_b\cap B_a}$
  and therefore $\abs{C_b\cap B_a}\le 2$, contradicting the assumption. 
  Thus we conclude that $\abs{B_c\cap C_a}\le 1$. 
  By Lemma~\ref{lem:meatiest},     at least two of   $G\setminus a$, $G\setminus b$, and $G\setminus c$ are \rc{3}{3}.
\end{proof}
\begin{lemma}\label{lem:extreme2}
  Let $((A_b,A_c),(B_a,B_c),(C_a,C_b))$ be a barrier of a prime \rc{3}{1} graph~$G$ with a triplet $\{a,b,c\}$
  such that 
  each of $\min(\abs{A_b},\abs{A_c})$, $\min(\abs{B_a},\abs{B_c})$, and $\min(\abs{C_a},\abs{C_b})$ is
  maximum.
  If $A_c\cap C_b\cap B_a=A_b\cap B_c\cap C_b=\emptyset$, then 
  at least two of   $G\setminus a$, $G\setminus b$, and $G\setminus c$ are \rc{3}{3}.
\end{lemma}
\begin{proof}
  Since all graphs up to $11$ vertices are \rc{3}{3}, we may assume that $\abs{V}\ge 13$.
  By Lemma~\ref{lem:meatiest}, we may assume that 
  $\abs{A_b\cap B_a}>1$, 
  $\abs{A_c\cap C_a}>1$, 
  $\abs{B_c\cap C_b}>1$, 
  $\abs{A_b\cap B_c}> 1$,
  $\abs{B_c\cap C_a}> 1$,
  $\abs{C_a\cap A_b}>1$, 
  $\abs{C_b\cap B_a}>1$,
  $\abs{B_a\cap A_c}>1$, and 
  $\abs{A_c\cap C_b}>1$.
  By Lemma~\ref{lem:sep2}\eqref{enum:sep2-1} and symmetry in $\{a,b,c\}$, we have 
  $\abs{A_b\cap B_c}, 
  \abs{B_c\cap C_a},
  \abs{C_a\cap A_b},
  \abs{C_b\cap B_a},
  \abs{B_a\cap A_c},
  \abs{A_c\cap C_b}\in\{2,3\}$.

  Since $G\setminus c$ is prime by Lemma~\ref{lem:primetriplet}
  and $\abs{A_c\cap C_b}\ge 2$, we have 
  $\rho(A_c\cap C_b,A_b\cup C_a)\ge 2$.
  Since $G\setminus b$ is prime by Lemma~\ref{lem:primetriplet}, $\rho(A_c,A_b-\{b\})\ge 2$ and therefore $\rho(A_c,A_b)=\rho(A_c,A_b-\{b\})$. 
  This implies that 
  \[ \rho(A_c\cap C_b,A_b\cup
  C_a-\{b\})=\rho(A_c\cap C_b,A_b\cup
  C_a)\ge 2.\] 

  Observe that $B_c\cap C_b=A_c\cap B_c\cap C_b=A_c\cap C_b$
  and therefore 
  \[ \rho(B_c\cap C_b, B_a\cup C_a)=\rho(A_c\cap C_b,A_b\cup
C_a-\{b\})\ge 2.\] 
  By Lemma~\ref{lem:sep2}\eqref{enum:sep2-3} and symmetry in $\{a,b,c\}$, we have 
  \( \abs{B_a\cap C_a}\le 3\). 
  Since $\abs{B_a\cap C_a}=1+\abs{A_b\cap B_a\cap C_a}+\abs{A_c\cap B_a\cap C_a}
  = 1+\abs{A_b\cap B_a\cap C_a}+\abs{B_a\cap A_c}-\abs{A_c\cap B_a\cap C_b}\ge 3+\abs{A_b\cap B_a\cap C_a}$, 
  we deduce that \[ A_b\cap B_a\cap C_a=\emptyset \text{ and } \abs{B_a\cap A_c}=2.\]

  This implies that $A_b\cap B_a= A_b\cap B_a\cap C_b = C_b\cap B_a$ and therefore
  \[ 
  \rho(A_b\cap B_a,A_c\cup B_c)=\rho(C_b\cap B_a, C_a\cup B_c -\{a\}).
  \] 
  Since $G\setminus a$ is prime by Lemma~\ref{lem:primetriplet}, 
  $\rho(C_b,C_a-\{a\})\ge 2$ 
  and therefore 
  $\rho(C_b,C_a-\{a\})=\rho(C_b,C_a)$,
  which implies that 
  $\rho(C_b\cap B_a,C_a\cup B_c-\{a\})=\rho(C_b\cap B_a,C_a\cup B_c)$
  and as $G-b$ is prime by Lemma~\ref{lem:primetriplet}, we have $\rho(C_b\cap B_a,C_a\cup B_c)\ge 2$. This implies that $\rho(A_b\cap B_a,A_c\cup B_c)\ge 2$ and so by Lemma~\ref{lem:sep2}\eqref{enum:sep2-3}, we have 
  \[ \abs{A_c\cap B_c}\le 3.\] 
  Note that $\abs{A_c\cap B_c}=1+\abs{A_c\cap B_c\cap C_a}+\abs{A_c\cap B_c\cap C_b} = 1+ \abs{A_c\cap B_c\cap C_a}+\abs{A_c\cap C_b}-\abs{A_c\cap C_b\cap B_a}\ge 3+\abs{A_c\cap B_c\cap C_a}$ and therefore \[ A_c\cap B_c\cap C_a=\emptyset \text{ and } \abs{A_c\cap C_b}=2.\] 

  Again, this implies that $A_c\cap C_a= A_c\cap B_a\cap C_a=A_c\cap B_a$ and therefore 
  \[ 
    \rho(A_c\cap C_a,A_b\cup C_b)=\rho(A_c\cap B_a, A_b\cup B_c-\{c\}).
  \] 
  Since $G\setminus c$ is prime, $\rho(B_a,B_c-\{c\})\ge 2$
  and so $\rho(B_a,B_c)=\rho(B_a,B_c-\{c\}$, which implies that 
  $\rho(A_c\cap B_a, A_b\cup B_c-\{c\})=\rho(A_c\cup B_a,A_b\cup B_c)$.
  Since $G\setminus a$ is prime and $\abs{A_c\cap B_a}\ge 2$, we have $\rho(A_c\cup B_a,A_b\cup B_c)\ge 2$, which implies that 
  $\rho(A_c\cap C_a,A_b\cup C_b)\ge 2$. 
  By Lemma~\ref{lem:sep2}\eqref{enum:sep2-3} and symmetry in $\{a,b,c\}$, we have 
  \(
    \abs{A_b\cap C_b}\le 3
  \). 
  Now, observe that 
  $\abs{A_b\cap C_b}=1+\abs{A_b\cap B_a\cap C_b}+\abs{A_b\cap B_c\cap C_b}=
  1+\abs{A_b\cap B_a}-\abs{A_b\cap B_a\cap C_a}
  +\abs{A_b\cap B_c\cap C_b}
  \ge 3+\abs{A_b\cap B_c\cap C_b}
  $ and therefore 
  \( 
    \abs{A_b\cap B_a}=2
  \). 

  Finally,  we deduce that 
  \[ 
    \abs{V}\le 2+\abs{A_b\cap B_a}+\abs{A_b\cap B_c}+\abs{A_c\cap B_a}+\abs{A_c\cap B_c}
    \le 2+2+3+2+3=12,
  \] 
  contradicting the assumption. 
\end{proof}

Finally, we are ready to finish the proof of
Proposition~\ref{prop:internally4prime}.
\begin{proof}[Proof of Proposition~\ref{prop:internally4prime}]
  By Lemma~\ref{lem:primetriplet}, $G\setminus a$, $G\setminus b$, and $G\setminus c$ are prime.
  Suppose that none of $G\setminus a$, $G\setminus b$, and $G\setminus c$ is \rc{3}{2}.
  Then there exists a barrier $((A_b,A_c),(B_a,B_c),(C_a,C_b))$ of $G$ with the triplet $\{a,b,c\}$.
  We may choose a barrier 
  such that 
  each of $\min(\abs{A_b},\abs{A_c})$, $\min(\abs{B_a},\abs{B_c})$, and $\min(\abs{C_a},\abs{C_b})$ is
  maximum. 

  Since all graphs up to $11$ vertices are \rc{3}{3}, we may assume that $\abs{V}\ge 13$.
  By Lemma~\ref{lem:meatiest}, we may assume that 
  $\abs{A_b\cap B_a}>1$, 
  $\abs{A_c\cap C_a}>1$, 
  $\abs{B_c\cap C_b}>1$, 
  $\abs{A_b\cap B_c}> 1$,
  $\abs{B_c\cap C_a}> 1$,
  $\abs{C_a\cap A_b}>1$, 
  $\abs{C_b\cap B_a}>1$,
  $\abs{B_a\cap A_c}>1$,
  $\abs{A_c\cap C_b}>1$,
  $\abs{A_c\cap B_c}>2$, 
  $\abs{A_b\cap C_b}>2$, 
  and 
  $\abs{B_a\cap C_a}>2$.
  By Lemma~\ref{lem:sep2}\eqref{enum:sep2-1} and symmetry in $\{a,b,c\}$, we have 
  $\abs{A_b\cap B_c}, 
  \abs{B_c\cap C_a},
  \abs{C_a\cap A_b},
  \abs{C_b\cap B_a},
  \abs{B_a\cap A_c},
  \abs{A_c\cap C_b}\in\{2,3\}$
  and 
  $\rho(A_b\cap B_c), 
  \rho(B_c\cap C_a),
  \rho(C_a\cap A_b),
  \rho(C_b\cap B_a),
  \rho(B_a\cap A_c),
  \rho(A_c\cap C_b)\le 2$.

  By Lemma~\ref{lem:lambda2'}, we may assume that $\abs{A_b\cap B_c\cap C_a}\le 1$ or $\abs{A_c\cap C_b\cap B_a}\le 1$.
  We may assume that 
  $\abs{A_c\cap C_b\cap B_a}\le 1$
  by exchanging $b$ with $c$ if necessary.

We may assume that $\abs{A_b\cap B_c\cap C_a}\ge 2$
by Lemmas~\ref{lem:sum1} and~\ref{lem:11}.
By Lemma~\ref{lem:lambda2},
we may assume that $\abs{(A_b\cap B_c)\cup (B_c\cap C_a)\cup (C_a\cap A_b)}\le 3$.
By replacing $(a,b,c)$ with $(b,c,a)$ or $(c,a,b)$ if necessary, 
we may assume that $C_a\cap A_b=A_b\cap B_c=A_b\cap B_c\cap C_a$
and therefore \[A_b\cap B_a\cap C_a=A_b\cap B_c\cap C_b=\emptyset.\]
By Lemma~\ref{lem:extreme2}, we may assume that \[ \abs{A_c\cap C_b\cap B_a}=1.\]
Observe that $2\le \abs{A_b\cap B_a}=\abs{A_b\cap B_a\cap C_b}
\le \abs{C_b\cap B_a}-\abs{A_c\cap C_b\cap B_a} \le 2$ and so 
\[ \abs{A_b\cap B_a}=2 \text{ and }\abs{C_b\cap B_a}=3.\] 
By Lemma~\ref{lem:commonextreme1}, we may assume that 
$\rho(A_b\cap B_a,A_c\cup B_c)\ge 2$ 
and so Lemma~\ref{lem:sep2}\eqref{enum:sep2-3} implies that
$\rho(A_c\cap B_c)\le 2$ and $\abs{A_c\cap B_c}\le 3$.
By the submodular inequality,
\[ 2+2\ge \rho(A_c\cap B_c)+\rho(A_c\cap C_b)\ge \rho(A_c\cap B_c\cap
  C_b)+\rho((A_c\cap B_c)\cup (A_c\cap C_b)).\] 
Since $G$ is prime and $\abs{A_c\cap B_c\cap C_b}
=\abs{B_c\cap C_b}-\abs{A_b\cap B_c\cap C_b}=\abs{B_c\cap C_b}\ge 2$, 
we have $\rho(A_c\cap B_c\cap C_b)\ge 2$.
Therefore, $\rho((A_c\cap B_c)\cup (A_c\cap C_b))\le 2$ and 
$\abs{(A_c\cap B_c)\cup (A_c\cap C_b)}=
\abs{A_c\cap B_c}+\abs{A_c\cap C_b\cap B_a} =4$, 
a contradiction to the assumption that $G$ is \rc{3}{1}.
\end{proof} 
This completes the proof of Theorem~\ref{thm:main}
as Proposition~\ref{prop:internally4prime}
implies Theorem~\ref{thm:main}.

\section{Excluded pivot-minors for rank-width $k$}\label{sec:rwd}
As an application, we will discuss  \emph{excluded pivot-minors of
  graphs of rank-width at most $k$}, which are pivot-minor-minimal graphs
whose   rank-width is larger than $k$.

First, we review the definition of
rank-width, defined by Oum and Seymour~\cite{OS2004}.
A tree  is \emph{subcubic} if every vertex has degree 1 or 3.
A \emph{rank-decomposition} of a graph $G$
is a pair $(T,\mu)$
of a subcubic tree $T$
and a bijection $\mu:V(G)\rightarrow \{x:\text{$x$ is a leaf of
    $T$}\}$.
For a rank-decomposition $(T,\mu)$ and each edge $e$ of $T$,
we pick a component $C_e$ of $T\setminus e$
and we let $A_e=\{v\in V(G): \mu(v)\in V(C_e)\}$.
Then the \emph{width} of an edge $e$ of $T$ in a rank-decomposition
$(T,\mu)$
is defined as $\rho_G(A_e)$.
Since $T\setminus e$ has exactly two components and
$\rho_G(A_e)=\rho_G(V(G)- A_e)$, the width of $e$ is well
defined.
We define the \emph{width} of a rank-decomposition $(T,\mu)$ as
the maximum width of an edge $e$  over all edges $e$ of $T$.
The \emph{rank-width} $\rwd(G)$ of a graph $G$ is the minimum width of $(T,\mu)$
over all rank-decompositions $(T,\mu)$ of the graph.
See \cite{Oum2016} for a survey on rank-width.

We can define $k$-branched sets recursively as follows:
A  set $B$ of vertices is \emph{$k$-branched}
if and only if
$\rho_G(B)\le k$
and
either $\abs{B}=1$ or
there exists a proper nonempty subset $B'$ of $B$
such that
both $B'$ and $B- B'$ are $k$-branched.
A set $A$ of vertices of a graph $G$ is \emph{titanic}
if for every partition $(A_1,A_2,A_3)$ of $A$,
there exists $i\in\{1,2,3\}$ such that
$\rho_G(A_i)\ge \rho_G(A)$.
The following lemma originates from Robertson and Seymour
\cite{RS1991}.
A proof can be found in~\cite[Lemma 5.1]{Oum2004}.
\begin{lemma}\label{lem:tangle}
  Let $G$ be a graph.
  If $A$  is a titanic set such that
  $\rho_G(A)\le k=\rwd(G)$,
  then $V(G)- A$ is $k$-branched.
\end{lemma}
The following lemma is a restatement  of~\cite[Lemma 5.3]{Oum2004}
and is a consequence of Proposition~\ref{prop:bixby} and Lemma~\ref{lem:tangle}.
We provide a proof.
\begin{lemma}\label{lem:bootstrap}
  Let $m$ be an integer.
  Let $G$ be a  graph
  such that  $\rwd(G)>m$ and
  $\rwd(G\setminus v)< \rwd(G)$ and $\rwd(G/v)< \rwd(G)$ for every vertex
  $v$ of $G$.
  If $G$ is \rc{m}{\ell}, then
  $G$ is \rc{(m+1)}{(6\ell+5m-5)}.
\end{lemma}
\begin{proof}
  Let $k=\rwd(G)-1$.
  Suppose that $G$ is not \rc{(m+1)}{(6\ell+5m-5)}.
  There exists a partition $(A,B)$ of $V(G)$ such that
  $\rho_G(A)< m+1$ and
  $\abs{A},\abs{B}\ge (6\ell+5m-5)+(m+1)$.
  We may assume that $B$ is not $k$-branched.

  Let $v\in A$. By Proposition~\ref{prop:bixby},
  $G\setminus v$ or $G/v$ is \rc{m}{(2\ell+m-1)}.
  We may assume that $G\setminus v$ is \rc{m}{(2\ell+m-1)}.

  In every partition $(X_1,X_2,X_3)$ of $A-\{v\}$,
  there exists $i\in \{1,2,3\}$ such that
  $\abs{X_i}\ge \slantfrac{(\abs{A}-1)}{3}$.
  We may assume that $\abs{X_1}\ge 2\ell+2m-1$.
  Since $G\setminus v$ is \rc{m}{(2\ell+m-1)},
  $\rho_{G\setminus v}(X_1)\ge m\ge \rho_{G\setminus
      v}(A-\{v\})$.
  Therefore $A- \{ v\}$ is a titanic set of $G\setminus v$.
  By Lemma~\ref{lem:tangle},  $B$ is $k$-branched in $G\setminus v$.
  Since $B$ is not $k$-branched in $G$,
  there must exist a subset $X$ of $B$
  such that $\rho_G(X)=\rho_{G\setminus v}(X)+1$.
  By Lemma~\ref{lem:gga},
  $\rho_{G\setminus v}(X)+\rho_G(B)\ge
    \rho_{G}(X)+\rho_{G\setminus v}(B)$ and therefore
  $\rho_G(B)=\rho_{G\setminus v}(B)+1$.
  But this is a contradiction because $G\setminus v$ is
  \rc{m}{(2\ell+m-1)}{} and $\rho_{G\setminus v}(B)<m$.
\end{proof}
Lemma~\ref{lem:bootstrap} can be directly applied to show that, for
each $k\ge 2$,
every pivot-minor-minimal graph having rank-width larger than $k$
is \rc{3}{5}.
In the following proposition, we prove that in fact they are
\rc{3}{2},
and our proof uses Proposition~\ref{prop:fully}.
\begin{proposition}\label{prop:conn}
  Let $G$ be a graph such that $\rwd(G)>2$.
  If $\rwd(G\setminus v)< \rwd(G)$ and $\rwd(G/v)< \rwd(G)$ for every vertex
  $v$ of $G$,
  then
  $G$ is prime \rc{3}{2}.
\end{proposition}
\begin{proof}
  Let $k=\rwd(G)-1\ge 2$.
  It is easy to see that $G$ is connected and therefore $G$ is
  \rc{1}{0}.
  By Lemma~\ref{lem:bootstrap}, $G$ is \rc{2}{0}.
  So $G$ is prime.

  Suppose that $G$ is not \rc{3}{2}.
  Let $(A,B)$ be a partition of $V(G)$ such that
  $\abs{A}, \abs{B}\ge 5$ and $\rho_G(A)\le 2$.
  If both $A$ and $B$ were $k$-branched, then $G$ would have
  rank-width $k$.
  Therefore, we may assume that $B$ is  not $k$-branched.
  We may assume that $B$ is minimal with such properties.

  If $\rho_G(A\cup \{w\})\le 2$ for some $w\in B$,
  then $(A',B')=(A\cup \{w\}, B- \{w\})$ is another partition.
  By assumption, either $\abs{B}=5$ or $B-\{w\}$ is $k$-branched.
  If $\abs{B}=5$, then $B-\{w\}$ is $k$-branched.
  If $B-\{w\}$ is $k$-branched, then
  $B$ is $k$-branched, contrary to the assumption that $B$ is not.
  Therefore $\rho_G(A\cup \{w\})>2$ for all $w\in B$.
  So $\rho_G(A)=2$ and $A$ is fully closed.

  Then by Proposition~\ref{prop:fully}, there exists $v\in A$ such that
  $G\setminus v$ or $G/v$ is prime.
  We may assume that $G\setminus v$ is prime, 
  because otherwise we replace $G$ with $G\pivot vw$ for some neighbor $w$ of $v$.
  Since $\abs{A}\ge 5$,
  in every 3-partition $(X_1,X_2,X_3)$ of $A- \{v\}$,
  there exists $i$ such that $\abs{X_i}\ge 2$.
  Since $G\setminus v$ is prime,
  $\rho_{G\setminus v}(X_i)\ge 2\ge \rho_{G\setminus v}(A-
    \{v\})$
  and therefore
  $A-\{v\}$ is titanic in $G\setminus v$.
  Since $G\setminus v$ has rank-width $k$,
  $B$ is $k$-branched in $G\setminus v$.
  Since $B$ is not $k$-branched in $G$,
  there exists a subset $X$ of $B$ such that
  $\rho_{G\setminus v}(X)=\rho_G(X)-1$.
  By Lemma~\ref{lem:gga},
  $\rho_{G\setminus v}(X)+\rho_G(B)\ge
    \rho_{G}(X)+\rho_{G\setminus v}(B)$ and therefore
  $\rho_G(B)=\rho_{G\setminus v}(B)+1$.
  But this is a contradiction because $G\setminus v$ is
  prime but $\rho_{G\setminus v}(B)= 1$.
\end{proof}
By combining Lemma~\ref{lem:bootstrap} with Proposition
\ref{prop:conn},
we obtain the following corollary, which is an improvement over the
previous theorem~\cite[Theorem 5.4]{Oum2004} stating that
such graphs are \rc{(k+1)}{(\frac{6^k-1}{5}-k)}{}
for every $k\in \{2,3,\ldots,\rwd(G)-1\}$.
This improvement is due to our improvement of the base case for the induction; we use the same inductive step.

\begin{corollary}\label{cor:conngen}
  Let $G$ be a graph such that $\rwd(G)> 2$.
  If $\rwd(G\setminus v)<\rwd(G)$ and $\rwd(G/v)<\rwd(G)$
  for every vertex $v$ of $G$,
  then $G$ is
  \rc{(k+1)}{(\frac{(7/12)\cdot  6^k-1}{5}-k)}{}
  for every $k\in \{2,3,4,\ldots, \rwd(G)-1\}$.
\end{corollary}
\begin{proof}
  We proceed by induction on $k$. 
  By Proposition~\ref{prop:conn}, 
  $G$ is \rc{3}{2}
  and so the statement is true for $k=2$. 
  Now let us assume that $2<k<\rwd(G)$ and  $G$ is \rc{k}{\ell} 
  for $\ell=\frac{(7/12)6^{k-1}-1}{5}-k$. 
  By Lemma~\ref{lem:bootstrap}, we deduce that 
  $G$ is \rc{k+1}{6\ell+5k-5}.
  This concludes the proof because 
  $ 6\ell+5k-5= 
  6\left(\frac{(7/12)\cdot  6^{k}-1}{5}-k\right)+5k-5=
    \frac{(7/12)6^{k+1}-1}{5}-(k+1)$. 
\end{proof}
Thanks to the improvement by Corollary~\ref{cor:conngen}, we can now
obtain an upper bound on the maximum number of vertices in such
graphs,
which is better than the previous upper bound $(6^{k+1}-1)/5$ in
\cite[Theorem 5.4]{Oum2004}.
\begin{theorem}\label{thm:upper}
  Let $k\ge 2$.
  If a graph $G$ has rank-width larger than $k$ but all of its proper
  pivot-minors have rank-width at most $k$, then
  $\abs{V(G)}\le \frac{ (7/12)\cdot 6^{k+1}- 1}{5}$.
\end{theorem}
\begin{proof}
  Let $n=\abs{V(G)}$.
  Let $N=\frac{(7/12)\cdot 6^{k}-1}{5}-k$.
  By Corollary~\ref{cor:conngen},
  $G$ is \rc{(k+1)}{N}.
  Let $v$ be a vertex of $G$.
  We may assume, by Proposition~\ref{prop:bixby}, that $G\setminus v$ is
  \rc{(k+1)}{(2N+k)}.
  Since $G\setminus v$ has rank-width at most $k$,
  there exists a partition $(A,B)$ of $V(G\setminus v)$ such that
  $\rho_{G\setminus v}(A)\le k$
  and $\abs{A},\abs{B}\ge (n-1)/3$.
  We may assume that $\abs{A}\le \abs{B}$.
  Since $G\setminus v$ is \rc{(k+1)}{(2N+k)},
  $\abs{A}<(2N+k)+(k+1)$.
  Therefore $n\le   6(N+k)+1=\frac{ (7/12)\cdot 6^{k+1}- 1}{5}$.
\end{proof}

When $k=2$, we can get a better bound $16$ by using 
Theorem~\ref{thm:main} on \rc{3}{2}{} graphs.
Previous bounds are $43$ by~\cite[Theorem 5.4]{Oum2004},
and $25$ by Theorem~\ref{thm:upper}. The author tried to compute the complete list of excluded pivot-minors for the class of graphs of rank-width at most $2$ by searching graphs up to $16$ vertices, but was not able to finish computation.

\begin{theorem}\label{thm:size}
  If a graph $G$ has rank-width $3$, but all of its proper pivot-minors
  have rank-width at most $2$, then
  $\abs{V(G)}\le 16$.
\end{theorem}
\begin{proof}
  By Proposition~\ref{prop:conn}, $G$ is prime \rc{3}{2}.
  Suppose that $\abs{V(G)}\ge 17$.
  By Theorem~\ref{thm:main},
  $G$ has a \rc{3}{3}{} pivot-minor $H$ such that
  $\abs{V(H)}=\abs{V(G)}-1\ge 16$.
  Since the rank-width of $H$ is $2$,
  there is a partition $(A,B)$ of $V(H)$ such that
  $\abs{A},\abs{B}\ge \abs{V(H)}/3>5$
  and $\rho_H(A)\le 2$, contradicting the assumption that $G$ is \rc{3}{3}.
\end{proof}
If we had a stronger connectivity condition on $G$,
then we could obtain a stronger bound from Proposition~\ref{prop:internally4prime}.

\begin{proposition}
  Suppose that a graph  $G$ has  rank-width $3$, but
  all of its proper pivot-minors have rank-width at most $2$.
  If $G$ is \rc{3}{1},
  then $\abs{V(G)}\le 14$.
\end{proposition}
\begin{proof}
  Suppose that
  there is a vertex $x\in G$ such that $G\setminus x$ or $G/x$ is
  \rc{3}{2}.
  We may assume that $G\setminus x$ is \rc{3}{2}{} because otherwise we can replace $G$ with $G\pivot xy$ for some neighbor $y$ of $x$.
  Since  $G\setminus x$ has rank-width 2, there is a partition
  $(A,B)$ of $V(G)- \{x\}$ such that $\rho(A,B)=2$
  and $\abs{A},\abs{B}\ge (\abs{V(G)}-1)/3$.
  Since $G\setminus x$ is \rc{3}{2},
  $\min(\abs{A},\abs{B})\le 4$. Therefore, $\abs{V(G)}\le 13$.
  We  may now assume that neither $G/x$ nor $G\setminus x$ is
  \rc{3}{2}.

  By Lemma~\ref{lem:4prime}, $G$ is not $3$-rank-connected
  and therefore there is a set $X=\{a,b,c\}$ of vertices such that
  $\rho(X)=2$.
  By Lemma~\ref{lem:triplet}, we may assume that $X$ is  a triplet
  of $G$, because otherwise we can apply pivoting.

  By Proposition~\ref{prop:internally4prime},
  we may assume that $G\setminus a$ is
  \rc{3}{3}.
  Let $(T,\mu)$ be a rank-decomposition of $G\setminus a$ having
  width~2.
  We orient every edge $uv$ of $T$ from $u$ to $v$
  if the partition given by $T\setminus uv$ has at most $5$ leaves
  in the component containing $v$.
  Since $T$ has more vertices than edges, there is a vertex $v$ of $T$
  such that all edges incident with $v$ are outgoing.
  If $v$ is a leaf, then $\abs{V(G)}\le 6$.
  So we may assume that $v$ is not a leaf and therefore there exists
  a partition $(A_1,A_2,A_3)$ of vertices of $G\setminus a$ such that
  $\abs{A_i}\le 5$ and $\rho(A_i)\le 2$ for all $i\in\{1,2,3\}$.

  If there exists $i$ such that $b,c\in A_i$,
  then $\rho(A_i\cup\{a\})\le 2$ because $\rho(\{b,c\},V(G)-
    X)=\rho(X,V(G)-X)$.
  Since $G$ is \rc{3}{1},
  either $\abs{A_i}+1\le 3$ or $\abs{V(G)- A_i}-1\le 3$.
  If $\abs{A_i}\le 2$, then $\abs{V(G)}\le 1+5+5+2=13$.
  If $\abs{V(G)- A_i}\le 4$, then $\abs{V(G)}\le 4+5=9$.
  We conclude that
  if there exists $i$ such that $b,c\in A_i$,
  then $\abs{V(G)}\le 13$.

  Now suppose that no $A_i$ contains both $b$ and $c$.
  We may assume that $b\in A_1$ and $c\in A_2$.
  Then $\rho(A_1\cup A_2,A_3)\le 2$ and therefore
  $\rho(A_1\cup A_2\cup\{a\},A_3)\le 2$.
  Since $G$ is \rc{3}{1},
  either $\abs{A_1\cup A_2}+1\le 3$ or
  $\abs{A_3}\le 3$.
  If  $\abs{A_1\cup A_2}\le 2$, then $\abs{V(G)}\le 1+2+5=8$.
  If $\abs{A_3}\le 3$, then $\abs{V(G)}\le 1+5+5+3=14$.
\end{proof}

\section{Conclusions}
We prove that every prime \rc{3}{2} graph with at least $10$ vertices 
has a prime \rc{3}{3} pivot-minor with one less vertices
and use this to show that every graph in the list of forbidden pivot-minors for the class of graphs of rank-width at most $2$ has at most $16$ vertices.

We would like to suggest a few open problems.
Hall~\cite{Hall2005} used Theorem~\ref{thm:hall} to prove the following.
\begin{theorem}[Hall~{\cite[Theroem 1.2]{Hall2005}}]\label{thm:hallmain}
  If $M$ is a $3$-connected matroid up to $3$-separators of size $5$, then 
  there is an element $x\in E(M)$ such that the cosimplification of $M\setminus x$
  or the simplification of $M/x$ is $3$-connected up to $3$-separators of size $5$, with a cardinality of $\abs{E(M)}-1$ or $\abs{E(M)}-2$.
\end{theorem}
Note that Theorem~\ref{thm:hallmain} ensures that obtained minors achieve the same connectivity requirement and so it is more useful for mathematical induction.
We would like to ask whether there is an analog of Theorem~\ref{thm:hallmain} for prime \rc{3}{3} graphs.
So far Theorem~\ref{thm:main} does not produce pivot-minors with the same connectivity. 

Another problem is to determine the complete list of forbidden pivot-minors for the class of graphs of rank-width $2$. The author wrote a C code to compute the complete list on a big computing cluster and yet $16$ is still too big to search all graphs. Up to $12$ vertices there are not too many forbidden pivot-minors according to the computation.

\subsection*{Acknowledgements}
The author would like to thank anonymous reviewers for 
their careful reading and useful suggestions. 
The author would also like to thank Duksang Lee and an anonymous reviewer 
for finding a mistake in the statement of Lemma~\ref{lem:sep} in the first version.

\end{document}